\documentclass[12pt]{amsart}
\usepackage{amsfonts}
\usepackage{amssymb, amsmath}
\usepackage[utf8x]{inputenc}
\usepackage{xspace}
\usepackage[inline,shortlabels]{enumitem}
\setlist[enumerate]{label={\upshape(\roman*)}}
\usepackage[a4paper,margin=3cm]{geometry}
\usepackage{mathtools}
\usepackage{mathrsfs}
\usepackage{stmaryrd}

\RequirePackage{my_tikz}
\tikzstyle{vertex}=[circle, draw, inner sep=0pt, font=\tiny, minimum size=12pt]
\newcommand{\vertex}{\node[vertex]}

\usepackage{booktabs}
\usepackage
{mathalfa}

\newtheorem{theorem}{Theorem}[section]
\newtheorem{corollary}[theorem]{Corollary}

\newtheorem{lemma}[theorem]{Lemma}
\newtheorem{proposition}[theorem]{Proposition}
\newtheorem{observation}[theorem]{Observation}

\theoremstyle{definition}
\newtheorem{example}[theorem]{Example}
\newtheorem{examples}[theorem]{Examples}
\newtheorem{definition}[theorem]{Definition}
\newtheorem{remark}[theorem]{Remark}

\newtheorem{que}[theorem]{Question}

\newcommand{\Z}{\mathbb{Z}}

\newcommand{\LL}{\mathcal{L}}
\newcommand{\FF}{\mathcal{F}}
\newcommand{\OO}{\mathcal{O}}

\newcommand{\rk}{\mathrm{rk \,}}
\newcommand{\gen}[1]{\langle #1 \rangle}

\newcommand{\dep}{\mathrm{dep}}
\newcommand{\Dep}{\mathrm{Dep}}
\newcommand{\DepCl}{\widehat{\mathrm{Dep}}}
\newcommand{\depl}{\mathrm{depl}}

\newcommand{\Fix}{\mathrm{Fix}}
\newcommand{\Aut}{\mathrm{Aut}}

\numberwithin{equation}{section}

\newcommand{\leqfg}{\leqslant_{\textsf{fg}}}
\newcommand{\leqff}{\leqslant_{\textsf{ff}}}

\usepackage{color}

\begin{document}

\title{Dependence over subgroups of free groups}

\author{Amnon Rosenmann}
\address[A.~Rosenmann]{Institut f\"{u}r Diskrete Mathematik, Technische Universit\"{a}t Graz, Austria}
\email{rosenmann@math.tugraz.at}

\author{Enric Ventura}
\address[E.~Ventura]{Departament de Matem\`{a}tiques, Universitat Polit\`{e}cnica de Catalunya and Institut de Matem\`{a}tiques de la UPC-BarcelonaTech, Catalonia}
\email{enric.ventura@upc.edu}
\thanks{Key words and phrases: dependence in free groups, algebraic extension in free groups, equations over free groups, Stallings foldings, Nielsen transformations}
\date{}

\maketitle

\begin{abstract}
Given a finitely generated subgroup $H$ of a free group $F$, we present an algorithm which computes $g_1,\ldots,g_m\in F$, such that the set of elements $g\in F$, for which there exists a non-trivial $H$-equation having $g$ as a solution, is, precisely, the disjoint union of the double cosets $H\sqcup Hg_1H\sqcup \cdots \sqcup Hg_mH$. Moreover, we present an algorithm which, given a finitely generated subgroup $H\leqslant F$ and an element $g\in F$, computes a finite set of elements of $H*\gen{x}$ that generate (as a normal subgroup) the ``ideal" $I_H(g) \unlhd H*\gen{x}$ of all ``polynomials" $w(x)$, such that $w(g)=1$. The algorithms, as well as the proofs, are based on the graph-theory techniques introduced by Stallings and on the more classical combinatorial techniques of Nielsen transformations. The key notion here is that of dependence of an element $g\in F$ on a subgroup $H$. We also study the corresponding notions of dependence sequence and dependence closure of a subgroup.
\end{abstract}

\section{Introduction}

The following is a basic question studied in classical algebra: given a \emph{field extension} $L / K$ ($L$ over $K$) and an element $\alpha \in L$, is $\alpha$ \emph{algebraic} over $K$, i.e., does there exist a non-trivial polynomial with coefficients in $K$, say $p(x)\in K[x]$, such that $p(\alpha )=0$, or, otherwise, is $\alpha$ \emph{transcendental} over $K$? Moreover, when $\alpha$ is algebraic over $K$, it is interesting to study the set of all polynomials over $K$ annihilating $\alpha$, namely $I_K(\alpha)=\{p(x)\in K[x] \mid p(\alpha)=0\}\subseteq K[x]$. As is well known, $I_K(\alpha)$ is a principal ideal of $K[x]$, $I_K(\alpha)=m_{\alpha}(x)K[x]$, whose unique monic generator $m_{\alpha}(x)$ is called the \emph{minimal polynomial} of $\alpha$ over $K$. An interesting algorithmic problem is then to compute $m_{\alpha}(x)$ in terms of $\alpha$.

Analogous questions can be asked and studied in the context of group theory, where the answers and algorithms turn out to be much more complicated. Here, we start with an \emph{extension of groups} $H\leqslant G$ and an element $g\in G$. The analog of the ``ring of polynomials over $K$" is the free product $H*\langle x\rangle \simeq H*\Z$, and the analog of the ``ideal of polynomials annihilating $\alpha$" is the normal subgroup $I_H(g) \unlhd H*\gen{x}$ of all ``polynomials" $w(x)\in H*\langle x\rangle$, such that $w(g)=1$. In the present paper, after some general considerations, we concentrate on extensions of free groups, $H\leqslant F$, where $H$ is finitely generated, and analyze both the set of elements in $F$ being algebraic over $H$, and the set of equations over $H$ satisfied by a given $g\in F$. We give the precise definitions below.

Let $G$ be a group and $H\leqslant G$. A univariate \emph{equation over $H$} (or an \emph{$H$-equation}, for short) is a ``polynomial equation" of the form
 \begin{equation}\label{eq:equation}
w(x)=1,
 \end{equation}
where $w(x)\in H*\gen{x}$, the free product of $H$ and the free abelian group of rank 1 generated by the \emph{variable} $x$. That is, $w(x)$ is an expression of the form
 \begin{equation}\label{eq:polynomial}
w(x)=h_0 x^{\varepsilon_1}h_1x^{\varepsilon_2}\cdots h_{d-1}x^{\varepsilon_d}h_d,
 \end{equation}
where $\varepsilon_i=\pm 1$ and $h_i\in H$. We assume that the equation is \emph{non-trivial}, i.e., $w(x)\in H*\gen{x}\setminus \{1\}$, and that it is in \emph{reduced form}, i.e., $\varepsilon_{i}=\varepsilon_{i+1}$ whenever $h_i=1$, for $i=1,\ldots ,d-1$. The \emph{degree} of the equation $w(x)=1$ (or of the polynomial $w(x)$), when written in reduced form, is $d=\sum_{i=1}^{d}\vert \varepsilon_i \vert$. The equation is \emph{balanced} if $\sum_{i=1}^{d} \varepsilon_i = 0$. It is standard to group together into higher exponents the various possible consecutive occurrences of $x$ with trivial elements in between; for example, $h_1 x^2 h_2 x^{-2}$ stands for $h_1 x 1 x h_2 x^{-1} 1 x^{-1}$, a balanced equation of degree 4. An element $g \in G$ is a \emph{solution} to equation~\eqref{eq:equation} if $w(g)=1$, that is, the element of $G$ resulting in substituting $g$ for $x$ in the expression~\eqref{eq:polynomial} is the trivial group element.

The area of Group Theory studying equations and their solutions is known as \emph{algebraic geometry over groups}; see~\cite{BMV99}: the elements of $H*\gen{x}$ (or, more generally, of $H*\gen{x_1,\ldots,x_n}$ in the case of multivariate equations) are the analog of non-commutative polynomials. Following the language from~\cite{BMV99}, a normal subgroup $I\unlhd H*\gen{x_1,\ldots,x_n}$ is called an \emph{ideal}, and the set of its common \emph{zeros} (i.e., the subset of $G^n$ of common solutions to the set of equations $\{ w(x)=1 \mid  w(x) \in I\}$) is an \emph{algebraic set}. There is a vast body of literature about the general problem of solving equations (univariate or multivariate) in a group $G$. A typical question is to decide, given an equation (or a system of equations), whether it has a solution in $G$ or not and, in the positive case, to describe the set of all such solutions and its structure. Two of the main results in this direction are Makanin's and Razborov's theorems, which analyze the case of free groups: the first one (see Makanin~\cite{Mak82}) solves the decidability part of the problem, while the second one (see Razborov~\cite{Raz87, Raz94}) provides a kind of compact algorithmic description of \emph{all} solutions, in case they exist. These are two very deep results, with intricate proofs, and having numerous important applications.

Here, we restrict ourselves to univariate equations and to free groups. In this particular case, the description of solution sets is much simpler: Bormotov, Gilman and Myasnikov \cite{BDM09} use reduction techniques from formal language theory to show that the solution set to a single equation of degree greater than 1 is either the whole group or a finite union of sets of the form $\{g_1 g_2^n g_3 \mid n\in \Z \}$. They also give a polynomial time algorithm for computing this solution set.

In contrast, we adopt here a kind of dual point of view. Instead of studying the solution set to a given equation in a group $G$, we fix a subgroup $H\leqslant G$ and are interested in the set of all solutions in $G$ of \emph{all} possible equations over $H$. In addition, given an element $g\in G$, we study the set of equations over $H$ satisfied by $g$, namely the ideal $I_H(g)=\{w(x)\in H*\gen{x} \mid w(g)=1\}\unlhd H*\gen{x}$: Is it trivial or not? And, if not, can we compute a finite set of generators for it (as a normal subgroup)? We shall answer and give algorithms for these and related questions, when the ambient group $G$ is a free group.

The fundamental notion here is that of dependence of an element $g\in G$ on a subgroup $H\leqslant G$, first introduced in Rosenmann~\cite{Ros01}:

\begin{definition}\label{def:dependent}
Let $G$ be a group, $H\leqslant G$ a subgroup, and $g\in G$ an element. We say that $g$ is \emph{dependent on $H$} if there exists a non-trivial $H$-equation $w(x)=1$ that is satisfied by $g$. More generally, $g$ is dependent on a subset $S \subseteq G$ if it is dependent on the subgroup $H=\gen{S}$. We denote by $\dep_G(H)$ the set of elements in $G$ that depend on $H$, and by $\Dep_G(H)=\gen{\dep_G(H)}$, the subgroup they generate, called the \emph{dependence} subgroup of $H$. When there is no risk of confusion, we shall delete the subscript $G$ from the notation.
\end{definition}

\begin{remark}
\label{rem:other_gps}
The notion of dependence, as given above, naturally suits free groups (as defined in \cite{Ros01} and as studied here). However, other possible definitions suit better in other contexts. For example, when $G$ is an abelian group and $1\neq H\leqslant G$ then every $g\in G$ satisfies the equation $hxh^{-1}x^{-1}=1$ (for each $h\in H$) and so, it is dependent on $H$. Here, we may define an element $g$ to be dependent on $H$ if it satisfies an equation of the form $x^nh=1$, $n\in \Z \setminus \{0\}$, $h\in H$. When $G$ is free abelian and $H$ is finitely generated then this definition is equivalent to $\rk(\gen{H,g}) \leqslant \rk(H)$. As seen below (Proposition~\ref{prop:g_depends_free}), the latter is also equivalent to Definition~\ref{def:dependent} in the case of free groups.
\end{remark}

\begin{remark}
Clearly, $H \subseteq \dep_G(H)\subseteq \Dep_G(H) \leqslant G$ but, in general, $\dep_G(H)$ is not a subgroup. Also, when $H\leqslant K\leqslant G$, we have $\dep_K(H)=K\cap \dep_G(H)$ and $\Dep_K(H)\leqslant K\cap \Dep_G(H)$, where the inclusion may be strict. We shall see examples of these situations in the free context; see Examples~\ref{ex:depHfree}.
\end{remark}

\begin{examples}\label{ex:depH}
\begin{enumerate}
\item If $g\in H$ then $g$ is dependent on $H$, satisfying the equation $g^{-1}x=1$ (also $gx^{-1}=1$) of degree 1.
\item If $g \notin H$ but $H^g\cap H\neq 1$ then $g$ is dependent on $H$, satisfying the balanced  equation $x^{-1}hxh'^{-1}=1$ of degree 2, where $h,h'\in H$ are such that $g^{-1}hg=h'$. It follows that if $\{1\}\neq H\unlhd G$ is a nontrivial normal subgroup of $G$ then $\dep(H)=G$. In particular, this is the case when $\{1\}\neq H$ is a subgroup of the center $Z(G)$ of $G$. It follows also that $N_G(H) \subseteq \dep(H)$, where $N_G(H) = \{ g \in G \mid g^{-1}Hg=H \}$ is the normalizer of $H$ in $G$.
\item If $H$ is of finite index in $G$ then $\dep(H)=G$: any $g\in G$ satisfies the equation $x^kh^{-1}=1$, where $k\geqslant 1$ and $h\in H$ are such that $g^k=h$.
\item $\dep(\{1\})$ is, precisely, the set of torsion elements of $G$. This follows from the fact that the only equations of positive degree over the trivial subgroup are those of the form $x^n=1$, $n\in \Z \setminus \{0\}$.
\end{enumerate}
\end{examples}

Clearly, the notion of ``$g$ being dependent on $H$" is the group theory analog of ``$\alpha$ being algebraic over the field $K$". At this point, we can already observe a relevant initial difference with classical field theory: the sum and product of two algebraic elements over $K$ are also algebraic over $K$ while, here, the product of two dependent elements on a subgroup $H$ may very well not be dependent on $H$; see Example~\ref{ex:depHfree}(i). So, one can expect the structure of the set $\dep(H)$ to be more complicated than its analog in the classical situation, where the algebraic elements over $K$ just form an intermediate field between $K$ and $F$. Here, $\dep (H)$ is not even a subgroup but, still, in the free group case, it will not be hard to understand its structure; see Theorem~\ref{thm:depgen} below.

In order to describe the set of elements $\dep(H)$, the following easy observation will be important: $\dep(H)$ is always the disjoint union of several (maybe infinite) $(H,H)$-double cosets:

\begin{observation}\label{obs:cosets}
Let $G$ be a group and let $H\leqslant G$. If $g\in \dep(H)$ then all the elements of the double coset $HgH$ depend on $H$ as well. Moreover, they all satisfy an $H$-equation of the same minimal degree.
\end{observation}

\begin{proof}
Let $w(x)=h_0 x^{\varepsilon_1}h_1x^{\varepsilon_2}\cdots h_{d-1}x^{\varepsilon_d}h_d=1$ be an $H$-equation satisfying $w(g)=1$. Then, for every $h,h'\in H$, we have that the new equation
 $$
w'(x) = h_0 (h^{-1}xh'^{-1})^{\varepsilon_1}h_1 (h^{-1}xh'^{-1})^{\varepsilon_2}\cdots h_{d-1} (h^{-1}xh'^{-1})^{\varepsilon_d}h_d = 1
 $$
satisfies $w'(hgh')=1$, showing that $hgh'\in \dep(H)$. Since (after appropriate cancellations) $w'(x)=1$ has the same or smaller degree than $w(x)=1$, it follows that all the elements of $HgH$ satisfy an $H$-equation of the same minimal degree.
\end{proof}

The following is another elementary result. Its converse, however, is not true in general (see Example~\ref{ex:depHfree}(iii) below).

\begin{observation}\label{obs:root}
Let $H\leqslant G$, $g\in G$, and $k\in \Z \setminus \{0\}$. If $g^k\in G$ is dependent on $H$ then $g$ is dependent on $H$ as well.
\end{observation}

\begin{proof}
If $h_0 x^{\varepsilon_1}h_1x^{\varepsilon_2}\cdots h_{d-1}x^{\varepsilon_d}h_d=1$ is a non-trivial $H$-equation satisfied by $g^k$ then, $h_0 x^{k\varepsilon_1}h_1 x^{k\varepsilon_2} \cdots h_{d-1}x^{k\varepsilon_d}h_d=1$ is a non-trivial $H$-equation satisfied by $g$.
\end{proof}

\medskip

Two related notions are those of a dependence-closed subgroup and the dependence closure operator. The latter appeared first in Rosenmann~\cite{Ros01} and it coincides with the notion of elementary-algebraic extension closure from Miasnikov--Ventura--Weil~\cite{MVW07}.

\begin{definition}
Let $H\leqslant G$. The subgroup $H$ is called \emph{dependence-closed} if $\Dep(H)=H$, i.e., if the only elements $g\in G$ which are solutions to non-trivial $H$-equations are those $g\in H$. For example, any free factor of $G$ is clearly dependence-closed.
\end{definition}

When $H$ is not dependence-closed then $\dep(H)$, the set of \emph{all} dependent elements
on $H$, strictly contains $H$. But then, when constructing the subgroup $\Dep(H)$ generated by these elements, new elements may emerge, which depend on $\Dep(H)$ but not on $H$. We can then iterate this process of executing the dependence operator and ask ourselves whether the process finally stabilizes, that is, do we reach a dependence-closed subgroup after finally many steps, or can the process continue indefinitely? In Proposition~\ref{prop:depl}, we show that when $H$ is a finitely generated subgroup of a free group then the process indeed stabilizes after finitely many steps. However, the number of steps may be arbitrary large even in the case of a subgroup of a free group of rank $2$, as shown in Example~\ref{ex:dep_length}, but it is still bounded by the total length of the generators of $H$ (see Proposition~\ref{prop:depl}).

\begin{definition}
Let the ascending sequence of subgroups $H_0\leqslant H_1\leqslant H_2\leqslant \cdots \leqslant G$ be defined by $H_0=H$ and $H_i=\Dep(H_{i-1})=\Dep^i(H)$, for $i\geqslant 1$. We define the \emph{dependence closure} of $H$ to be the subgroup
 $$
\DepCl(H)=\bigcup_{i\geqslant 0} H_i \leqslant G.
 $$
If this ascending sequence stabilizes after finitely many steps, we define the \emph{dependence length} of $H$ to be $\depl(H)=\min \{j\geqslant 0 \mid H_j=\DepCl(H)\}$; otherwise, we let $\depl(H)=\infty$.
\end{definition}

\begin{observation}
For any subgroup $H\leqslant G$, $\DepCl(H)$ is the smallest dependence-closed subgroup of $G$ containing $H$.
\end{observation}

\begin{proof}
Let $g\in G$ be dependent on $\DepCl(H)$, i.e., $g$ is a solution to some non-trivial $\DepCl(H)$-equation $w(x)=h_0x^{\varepsilon_1}h_1\cdots x^{\varepsilon_d}h_d=1$. Since $h_0,h_1,\ldots ,h_d\in \DepCl(H)$, there exists $m\geqslant 0$, such that $h_0,h_1,\ldots ,h_d\in \Dep^m(H)$. Thus, $w(x)=1$ is also a non-trivial $\Dep^m(H)$-equation and hence, $g\in \Dep^{m+1}(H)\leqslant \DepCl(H)$. Therefore, $\DepCl(H)$ is dependence-closed. The rest of the statement is clear.
\end{proof}

The following two propositions are straightforward, and we leave their proofs to the reader.

\begin{proposition}
Let $H\leqslant K\leqslant G$. If $H$ is dependence-closed in $K$, and $K$ is dependence-closed in $G$ then $H$ is dependence-closed in $G$. \hfill $\qed$
\end{proposition}

\begin{proposition}
Let $H\leqslant G$ be a dependence-closed subgroup of a group $G$. Then:
\begin{enumerate}
\item $H$ is malnormal in $G$: $H^g \cap H=1$ for every $g\in G\setminus H$;
\item $H$ is pure in $G$ (also called root-closed, radical-closed): if $g^n \in H$, $n\neq 0$, then $g\in H$;
\item if $H\leqslant K\leqslant G$ and $H \ne K$ then $|K:H|=\infty$;
\item if $K$ is a free factor of $H$  then $K$ is dependence-closed in $G$;
\item if $K\leqslant G$ is also dependence-closed in $G$ then so is $H\cap K$. \hfill $\qed$
\end{enumerate}
\end{proposition}

\medskip

In the following and along the rest of the paper, when referring to morphisms, we adopt the notation with arguments on the left: $x\mapsto (x)\varphi=x\varphi$, and compositions denoted accordingly: $x\mapsto x\varphi\mapsto x\varphi\phi$.

Let us now look at $H$-equations from a different perspective. When $H\leqslant G$ and $g\in G$, we can consider the group homomorphism
 \begin{eqnarray}\label{eq:phi}
\varphi_{H,g} \colon H*\gen{x} & \to & G, \nonumber \\ h & \mapsto & h, \; \forall h\in H, \\ x & \mapsto & g, \nonumber
 \end{eqnarray}
which is well defined by the universal property of free products. Then,
 $$
(w(x))\varphi_{H,g} =w(g).
 $$
So, in this setting, $g$ is a solution to $w(x)=1$ if and only if $w(x)\in \ker \varphi_{H,g}$. Thus, the set of $H$-equations satisfied by $g$ is exactly
 $$
\{w(x)=1 \mid w(x) \in \ker \varphi_{H,g} \},
 $$
and $\ker \varphi_{H,g}$ is a normal subgroup of $H*\gen{x}$. Hence, a natural way to describe these $H$-equations is by giving a set of generators for the ideal $I_H(g)=\ker \varphi_{H,g}$ as a normal subgroup of $H*\gen{x}$. This is what we do in the free case: we show that $I_H(g)$ is always finitely generated (as normal subgroup) and we provide an algorithm computing a set of generators for it; see Theorem~\ref{thm:eqbasis}.

\subsection{Organization of the paper}

After this section, where the basic notions and definitions corresponding to dependence on subgroups in general groups have been introduced, we restrict ourselves to free groups $F$ for the rest of the paper. In Section~\ref{sec:Dep_free_groups} we exhibit basic results for free groups. In Section~\ref{sec:Stallings} we present a brief description of the well-known Stallings techniques for studying subgroups of a free group, which play a central role in the paper. Then, in Section~\ref{sec:all_dependent_elements}, we examine closely the way we can use Stallings foldings in order to provide an explicit description of the set of elements that depend on a finitely generated subgroup $H\leqslant F$. This set is proved to have the form of a finite disjoint union of double cosets $\dep(H)= Hg_0H\sqcup Hg_1H\sqcup \cdots \sqcup Hg_nH$, where $g_0=1$ and the $g_i$'s are computable and generate $\Dep(H)$; see Theorem~\ref{thm:depgen}. In Section~\ref{sec:equations} we present two alternative algorithms for computing a finite set of normal generators for the ideal (i.e., normal subgroup) $I_H(g) \unlhd H*\gen{x}$ consisting of those polynomials $w(x)$ over $H$ satisfying $w(g)=1$, for a given fixed element $g\in F$; see Theorem~\ref{thm:eqbasis}. The first algorithm is based on Nielsen transformations and the second one on Stallings graph-theory techniques providing, respectively, an algebraic and a geometric reason (in a way, complementary to one another) for such equations to exist. Finally, in Section~\ref{sec:dep_closure} we study further the notion of a dependence-closed subgroup, and the corresponding closure operator within the context of free groups $F$, including an example of a subgroup $H\leqslant F_2$ of an arbitrarily long dependence length. We close the paper highlighting a few interesting questions for further research.

\bigskip

A preliminary version of this paper was uploaded to arXiv on July 2021; see~\cite{RV21}. A year later, in the process of writing the present final version, we have been informed of a recent further development in this interesting topic: by applying as well Stallings techniques, Dario Ascari~\cite{As22} has gone further and has analyzed the set of degrees of the polynomials in the ideal $I_H(g)$. Answering one of our questions from~\cite{RV21}, he presented an algorithm that computes a non-trivial polynomial in $I_H(g)$ of a minimal degree. Also, he exhibited an interesting example of a finitely generated subgroup $H\leqslant F$, whose corresponding ideal $I_H(g)$ \emph{cannot} be generated by its polynomials of minimal degree, even as a normal subgroup (providing another intriguing difference with the more classical field theory).

\section{Dependence in free groups}\label{sec:Dep_free_groups}

For the rest of the paper we concentrate on free groups and study the above notions in this setting, both from the algebraic and the algorithmic points of view. We adopt the following notation: a free group is generally denoted by $F$; also by $F_r$, if we want to emphasize its rank, $\rk(F_r)=r$; and by $F(A)$, if we want to emphasize a basis (a set of free generators) $A\subseteq F$ for it. The notations $H\leqfg F$ and $H\leqff F$ mean that $H$ is a finitely generated, respectively, a free factor of $F$.

The next proposition follows from works by Nielsen and Schreier (see Rosenmann~\cite[Lemma 1.1]{Ros01}, or Miasnikov--Ventura--Weil~\cite[Prop. 3.13]{MVW07}) and provides equivalent definitions for dependence of elements in the setting of free groups.

\begin{proposition}\label{prop:g_depends_free}
Let $H\leqslant F$ and $g\in F$. The following statements are equivalent:
\begin{enumerate}
\item $g$ does not depend on $H$;
\item the morphism $\varphi_{H,g}\colon H*\gen{x}\to F, h \mapsto h, x\mapsto g$ is injective;
\item $H$ is a proper free factor of $\gen{H,g}$;
\item $H$ is contained in a proper free factor of $\gen{H,g}$;
\item $\rk(\gen{H',g})=\rk(H')+1$ for every finitely generated subgroup $H' \leqslant H$;
\item $\rk(\gen{H',g})>\rk(H')$ for every finitely generated subgroup $H' \leqslant H$. \qed
\end{enumerate}
\end{proposition}

In view of Proposition~\ref{prop:g_depends_free}, it is quite easy to decide, given $H\leqfg F$ and $g\in F$, whether $g$ is dependent on $H$ or not: we only need to compute the rank of $\gen{H,g}$. In case $g$ is dependent on $H$, we can even find an explicit non-trivial $H$-equation $w(x)=1$, such that $w(g)=1$, by a brute force algorithm, enumerating all possible $H$-equations and checking for which one is $g$ a solution. In the following sections we tackle more challenging problems:
\begin{enumerate}
\item in Section~\ref{sec:all_dependent_elements} we provide an effective algorithm which, given $H\leqfg F$, computes a finite set of generators for $\Dep(H)$;
\item in Section~\ref{sec:equations} we present two independent algorithms (both faster and more interesting than just the brute force one) which, given $g\in \dep(H)$, results in a representation of \emph{all} the $H$-equations having $g$ as a solution.
\end{enumerate}

\begin{examples}\label{ex:depHfree}
Let $F_2=F(\{a,b\})$ be a free group of rank 2.
\begin{enumerate}
\item For $H=\gen{a^2, b^2}$, we have $a,b\in \dep(H)$ (being solutions to the $H$-equations $x^2a^{-2}=1$ and $x^2b^{-2}=1$, respectively) and so, $\Dep(H)=F_2$. However, $ab\notin \dep(H)$ since $\{a^2, b^2, ab\}$ is a freely independent subset of $F_2$. It follows that $\dep(H)$ is not a subgroup and thus, it is strictly contained in $\Dep(H)$.
\item When $H=\gen{a^2, b^2}$ and $K=\gen{a^2, b^2, ab}\leqslant F_2$ then $H\leqff K\leqslant F_2$. It follows that $\Dep_K(H)=H$ and $K\cap \Dep_{F_2}(H)=K\cap F_2 =K$, so that the inclusion $\Dep_K(H) \leqslant K\cap \Dep_{F_2}(H)$ is strict.
\item For $H=\gen{b, aba^{-1}}\leqslant F_2$, we have $a\in \dep(H)$, whereas $a^r\not\in \dep(H)$ for $|r|\geqslant 2$.
\item Let $F$ be a free group and $g\in F$ not being a proper power. Then, for every $0\neq r\in \Z$, $\dep(\gen{g^r})=\Dep(\gen{g^r})=\gen{g}$. In particular, $\gen{g^r}$ is dependence-closed if and only if $r=\pm 1$, i.e., the dependence-closed cyclic subgroups are, precisely, the maximal ones.
\item Let $F$ be a free group, $g\in F$, and $H,K\leqslant F$, such that $F=H*\gen{g}$ and $H<K$ (strict inclusion). Then, $\Dep(K)=F$ because $g$ is dependent on $K$.
\end{enumerate}
\end{examples}

When $H\leqfg F$ then $\Dep(H)\leqfg F$ and, moreover, $\rk(\Dep(H))\leqslant \rk(H)$. This follows from results in Rosenmann~\cite{Ros01} and Miasnikov--Ventura--Weil~\cite{MVW07}; for completeness, we offer here a direct and elementary proof.

\begin{lemma}\label{lem:sequence}
Let $H_0 \leqfg H_1\leqfg H_2 \leqfg \cdots$ be a non-decreasing sequence of finitely generated subgroups of a free group $F$, satisfying $\rk (H_i)\geqslant \rk (H_{i+1})$, for $i\geqslant 0$. Let $H=\bigcup_{i\geqslant 0} H_i$. Then $\rk (H)=\lim_{i\to \infty} \rk (H_i)$ and, furthermore, there exists $m$, such that $H=H_i$ for all $i\geqslant m$.
\end{lemma}

\begin{proof}
Let $r_i =\rk(H_i)$, for each $i\geqslant 0$, and let $r_H=\rk(H)$. The sequence $(r_i)_{i\geqslant 0}$ of natural numbers is monotone decreasing and therefore eventually constant: there exists $r\geqslant 0$ and $m'\geqslant 0$ such that $r_i=r$, for all $i\geqslant m'$. Suppose now that $r_H >r$. It follows that there exists $H'\leqff H$ of rank $r+1$. But then, $H'\leqslant H_{m''}$, for some $m''$, and, in fact, $H'\leqff H_i$, for all $i\geqslant m''$, which is in contradiction to $\rk(H_i)=r$ for all $i\geqslant m'$. Therefore, $r_H\leqslant r$ and there exists $m$, such that $H_m$ contains a generating set of $H$, which implies that $H=H_i$, for all $i\geqslant m$, and $r_H=r$.
\end{proof}

\begin{proposition}\label{prop:rankDep}
Let $H\leqfg F$. Then, $\rk(\Dep(H))\leqslant \rk(H)$.
\end{proposition}

\begin{proof}
Being finitely generated, $H$ is contained in a finitely generated free factor of $F$ (even in the case that $F$ is not countably generated). Therefore, there are only countably many elements in $F$ depending on $H$, say $(g_i)_{i \geqslant 1}$. Let $(H_i)_{i \geqslant 0}$ be the sequence of ascending subgroups of $F$ defined by $H_0 =H$ and $H_{i+1}=\gen{H_i, g_{i+1}}$, for $i>0$. Since $g_{i+1}\in \dep(H)$ then $g_{i+1}\in \dep(H_i)$ and hence, by Proposition~\ref{prop:g_depends_free}(vi), $\rk (H_{i+1}) \leqslant \rk (H_i)$, for each $i$. The result then follows by Lemma~\ref{lem:sequence}.
\end{proof}

\medskip

\begin{proposition}\label{prop:properties}
Let $H,K\leqslant G$. Then,
\begin{enumerate}
\item $\Dep(H\cap K)\leqslant \Dep(H)\cap \Dep(K)$ and the inclusion may be strict;
\item $\Dep(\gen{H, K})\geqslant \gen{\Dep(H), \Dep(K)}$ and the inclusion may be strict.
\end{enumerate}
\end{proposition}

\begin{proof}
(i). The inclusion is immediate by definition. As an example for a strict inclusion, let $F=F(\{a,b,c\})$, $H=\gen{a^{-1}ba,b}$ and $K=\gen{a^{-1}ca,c}$. Then $H\cap K=1$ and so, $\Dep(H\cap K)=1$. However, $\Dep(H)=\gen{a,b}$, $\Dep(K)=\gen{a,c}$, hence $\Dep(H)\cap \Dep(K)=\gen{a}$.

(ii) The inclusion is clear. As an example for a strict inclusion, let $F=F(\{a,b\})$, $H=\gen{a^2b}$ and $K=\gen{a^4b}$. By Example~\ref{ex:depHfree}(iv), $\Dep(H)=H$ and $\Dep(K)=K$. But $\gen{H,K}=\gen{a^2, b}$ and so, $\Dep(\gen{H,K})=F$, which strictly contains $\gen{a^2, b}=\gen{\Dep(H), \Dep(K)}$.
\end{proof}

\section{Stallings graph-theory techniques}\label{sec:Stallings}

Let $A=\{a_1, \ldots ,a_r\}$ be an alphabet of $r$ letters, let $A^{\pm}=\{a_1, \ldots ,a_r, a_1^{-1}, \ldots ,a_r^{-1}\}$ be its formal involutive closure, and let $F(A)$ be the free group on $A$ (formally, $F(A)$ is the free monoid $(A^{\pm})^*$ on $A^{\pm}$, modulo the equivalence relation generated by the elementary reductions $a_ia_i^{-1}\sim a_i^{-1}a_i \sim 1$). In 1983, elaborating on previous ideas by several authors, Stallings~\cite{Sta83} established the notion of so-called \emph{Stallings $A$-automata}: oriented graphs (allowing loops and parallel edges) with labels from $A^{\pm}$ at the edges, being \emph{involutive} (i.e., for every edge $e$ from $p$ to $q$ with label $a\in A$ there is an edge $e^{-1}$ from $q$ to $p$ labelled $a^{-1}$; $e$ and $e^{-1}$ are said to be \emph{inverse} to each other), with a selected vertex called the \emph{basepoint} (denoted here $\bp$), and being connected, \emph{deterministic} (there are no two different edges with the same label coming from, or into, the same vertex) and \emph{trim} (every vertex appears in some reduced closed path through the basepoint). A (nontrivial) \emph{path} of \emph{length} $n\geqslant 1$ is a concatenation of edges $\gamma=e_1^{\varepsilon_1}\cdots e_n^{\varepsilon_n}$, such that $\tau e_i^{\varepsilon_i}=\iota e_{i+1}^{\varepsilon_{i+1}}$ (the terminal vertex of an edge is the initial vertex of the following edge), for $i=1,\ldots ,n-1$. The label of $\gamma$ is $\ell(\gamma)=(\ell(e_1))^{\varepsilon_1} \cdots (\ell(e_n))^{\varepsilon_n}$, the product of the labels of the edges, understood as an element of $F(A)$. The path is \emph{reduced} if it has no backtracking, i.e., if $\varepsilon_i= \varepsilon_{i+1}$ whenever $e_i=e_{i+1}$. Note, that in the deterministic case, a path $\gamma$ is reduced if and only if its label $\ell(\gamma)$ is a reduced word of $(A^{\pm})^*$. A path beginning and ending at vertex $p$ is called a \emph{closed path at $p$} (a \emph{$p$-path}, for short). For later convenience, given a vertex $p$ in a deterministic $A$-automaton $\Gamma$ and given a reduced word $u\in F(A)$, we define $pu$ to be the end of the unique possible path in $\Gamma$ labelled $u$ and starting at $p$; in case no such path exists, $pu$ remains undefined.

For the sake of brevity, we may only mention and depict the positive subautomaton of $\Gamma$ (i.e., that formed by the edges labelled by letters in $A$), denoted $\Gamma^+$, bearing in mind that, for each of such edges, $\Gamma$ also contains its inverse (even if not depicted or mentioned). The \emph{degree} of a vertex in $\Gamma$ is the number of edges of $\Gamma^+$ incident to it.

With this notion, Stallings~\cite{Sta83} established a bijection between the set of (free) subgroups of $F(A)$, and the set of isomorphism classes of Stallings $A$-automata.

\begin{theorem}[Stallings, \cite{Sta83}]\label{thm:Stallings}
The following is a bijection:
 $$
\begin{array}{rcl}
St \colon \{H\leqslant F(A)\} & \longrightarrow & \{\mbox{isom. classes of Stallings $A$-automata}\} \\ H & \mapsto & \Gamma_{A}(H), \\ \LL(\Gamma) & \mapsfrom & \Gamma.
\end{array}
 $$
Furthermore, $H\leqfg F_A$ if and only if $\Gamma_A(H)$ is finite; in this case, both directions are computable. \qed
\end{theorem}

The map to the left consists of reading the \emph{language} subgroup $\LL(\Gamma)\leqslant F(A)$ of a given Stallings $A$-automaton $\Gamma$, i.e., the set of labels of reduced $\bp$-cycles (cycles that contain the basepoint) in $\Gamma$. As for the map to the right, it assigns to each subgroup $H\leqslant F(A)$ its \emph{Stallings} $A$-automaton $\Gamma_A(H)$, being the core of the \emph{Schreier coset graph} of $H$ with respect to the basis $A$ of $F(A)$, and with the trivial coset taken as the basepoint, $\bp=H1$. By a \emph{core} of an $A$-automaton $\Gamma$, denoted $c(\Gamma)$, we mean its largest trim subautomaton. It follows that $\Gamma$ is trim if and only if $c(\Gamma)=\Gamma$.
One can think of the Schreier coset graph of $H\leqslant F(A)$ with respect to $A$ as the covering graph of the $A$-bouquet $\FF(A)$ corresponding to the subgroup $H$; and its \emph{Stallings graph} $\Gamma_A(H)$ as its core, i.e., the result of trimming all the hanging trees not containing the basepoint. For more details, see~\cite{DV21, KM02, MVW07, Sta83}.

In this paper, the important part of the bijection in Theorem~\ref{thm:Stallings} is when restricted to finitely generated subgroups on one side, and to finite Stallings $A$-automata on the other. In this case, it becomes algorithmic friendly, i.e., there are fast algorithms for computing both directions, which we summarize in the following sentences. Let us also mention that, in the finite and connected case, it is easy to see that an $A$-automaton $\Gamma$ is trim if and only if no vertex in $\Gamma$ has degree 1, except, possibly, for $\bp$; moreover, in this case, one can obtain $c(\Gamma)$ from $\Gamma$ by applying finitely many times the trim operation: removing a vertex of degree one different from $\bp$ (together with the incident edge).

In order to compute the map to the left, assuming that $\Gamma$ is finite, we first construct a spanning tree $T$ of $\Gamma$ by removing a set of edges $E'=\{e\in E\Gamma^+\setminus ET\}$. Then, for each $e\in E'$, consider the group element $h_e=\ell(T[\bp, \iota e]\cdot e\cdot T[\tau e, \bp])\in F(A)$, where $T[p,q]$ stands for the unique reduced path in $T$ from vertex $p$ to vertex $q$ (the \emph{$T$-geodesic} from $p$ to $q$). It is easy to see that, when $\Gamma$ is deterministic, this is a free basis for $\LL(\Gamma)\leqslant F(A)$, whose rank then coincides with the first Betti number (cyclomatic number) of $\Gamma$, namely $\rk(\LL(\Gamma))= 1-|V\Gamma|+ |E\Gamma|$ (in general, without determinism, we can only say that $\{h_e \mid e\in E'\}$ generates $\LL(\Gamma)$ and so, $\rk(\LL(\Gamma))\leqslant 1-|V\Gamma|+|E\Gamma|$).

In order to compute the map to the right, we need to use a key result in this setting: \emph{``for every $H\leqslant F(A)$ there exists a unique $A$-automaton (up to isomorphism) which is deterministic, trim, and has language $H$"}; it is called the \emph{Stallings $A$-automaton} for $H$, and is denoted $\Gamma_A(H)$. Given a set $W=\{h_1,\ldots ,h_n\}$ of reduced words on $A$, we can construct the Stallings $A$-automaton $\Gamma_A(H)$ for the subgroup $H=\gen{h_1, \ldots ,h_n}\leqslant F(A)$ by applying the following procedure. First, for each $h_i$, $i=1,\ldots ,n$, we form the circular graph (called \emph{petal}) whose $A$-label, when starting at the basepoint, spells the word $h_i$ (or its inverse if traveled in the opposite direction). Then, we form the \emph{wedge} of these circuits by identifying the $n$ basepoints into a single one, denoted $\bp$. The resulting $A$-automaton is called the \emph{flower automaton}, denoted $\FF(W)$. Note, that $\LL(\FF(W))=H$. In order to gain determinism, we apply a finite series of \emph{elementary foldings}: whenever two edges sharing the same initial (resp., terminal) vertex have the same label, we identify them into a single edge, together with their terminal (resp., initial) vertices. When the two terminal as well as the two initial vertices are the same, the elementary folding is called \emph{closed} (or \emph{non-homotopic}) (see Figure~\ref{fig:foldings}(c),(d)) and the first Betti number of the automaton reduces by one; otherwise, the folding is called \emph{open} (or \emph{homotopic}) (see Figure~\ref{fig:foldings}(a),(b)), and the first Betti number remains unchanged. Note, that in a closed folding we identify two edges, whereas in an open folding we identify two edges as well as two vertices. In both cases, the language of the automaton remains unchanged. Of course, after performing an elementary folding, new possible foldings may emerge, but since at each step the number of edges strictly decreases, the process terminates after a finite number of steps. By construction, the resulting $A$-automaton is deterministic and has language $H$. It is also easy to see that it is trim (each vertex of the flower automaton, except possibly for the basepoint, has two incident edges of different labels, and they remain incident to the vertex throughout all the folding process). By the previous key result, the output of this process must be $\Gamma_A(H)$. In particular, this implies that the result of the folding process is independent of the specific sequence of foldings applied, and even of the set of generators for $H$ we started with: it only depends on the subgroup $H\leqslant F(A)$ (and on the ambient basis $A$ chosen to work with).

Touikan~\cite{Tou06} gave a more refined algorithm to compute $\Gamma_A(H)$ in time $\OO(n\log^*(n))$, where $n$ is the sum of the lengths of the given generators for $H$, and $\log^*(n)$ is the minimal number of successive applications of the log operator to $n$ until reaching 1 or below.

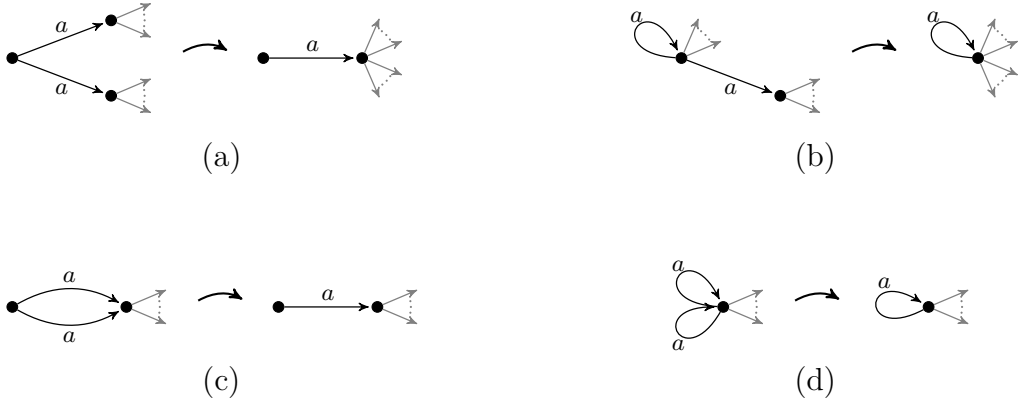
\begin{figure}
\centering
\begin{tikzpicture}[shorten >=1pt, node distance=1cm and 1.5cm, on grid,>=stealth',scale=1.1]
\tikzset{every loop/.style={min distance=10mm,in=0,out=60,looseness=5}}
\begin{scope}
	\node[state] (1) {};
	\node[state] (2) [above right = 0.5 and 1.3 of 1] {};
   		\node[] (21) [above right = 0.3 and 0.7 of 2] {};
   		\node[] (22) [below right = 0.3 and 0.7 of 2] {};
   \node[state] (3) [below right = 0.5 and 1.3 of 1] {};
   		\node[] (31) [above right = 0.3 and 0.7 of 3] {};
   		\node[] (32) [below right = 0.3 and 0.7 of 3] {};
   \path[->]
        (1) edge[] node[pos=0.5,above] {$a$} (2)
            edge[] node[pos=0.5,below] {$a$} (3)
        (2) edge[gray] (21)
            edge[gray] (22)
        (3) edge[gray] (31)
            edge[gray] (32);
    \node[] (i) [right = 2.1 of 1]{};
    \node[] (t) [right = 0.9 of i] {};
    \path[->]
        (i) edge[thick, bend left] (t);
	\node[state] (1') [right= 0.3 of t] {};
	\node[state] (2') [right =  1.3 of 1'] {};
		\node[] (21') [above right = 0.7 and 0.3 of 2'] {};
   		\node[] (22') [above right = 0.3 and 0.7 of 2'] {};
   		\node[] (31') [below right = 0.3 and 0.7 of 2'] {};
   		\node[] (32') [below right = 0.7 and 0.3 of 2'] {};
   \path[->]
        (1') edge node[pos=0.5,above] {$a$} (2')
        (2') edge[gray] (21')
             edge[gray] (22')
             edge[gray] (31')
             edge[gray] (32');
	\foreach \n [count=\count from 1] in {1,2,3}{
       \node[dot,gray] (2d\n) at ($(2)+(-20+\count*10:0.4cm)$) {};}
	\foreach \n [count=\count from 1] in {1,2,3}{
       \node[dot,gray] (2d\n) at ($(3)+(-20+\count*10:0.4cm)$) {};}
	\foreach \n [count=\count from 1] in {1,2,3}{
       \node[dot,gray] (2d\n) at ($(2')+(24+\count*10:0.4cm)$) {};}
	\foreach \n [count=\count from 1] in {1,2,3}{
       \node[dot,gray] (2d\n) at ($(2')+(-63+\count*10:0.4cm)$) {};}
	\draw (2.5,-1.2) node {(a)};
\end{scope}
\begin{scope}[xshift=8cm]
	\node[state] (1) {};
		\node[] (11) [above right = 0.7 and 0.3 of 1] {};
		\node[] (12) [above right = 0.3 and 0.7 of 1] {};
	\node[state] (3) [below right = 0.5 and 1.3 of 1] {};
		\node[] (31) [above right = 0.3 and 0.7 of 3] {};
		\node[] (32) [below right = 0.3 and 0.7 of 3] {};
	\path[->]
		(1) edge[loop,rotate=120] node[pos=0.5,above] {$a$} (1)
			edge[] node[pos=0.5,below] {$a$} (3)
			edge[gray] (11)
			edge[gray] (12)
		(3) edge[gray] (31)
			edge[gray] (32);
	\node[] (i) [right = 2.1 of 1]{};
	\node[] (t) [right = 0.9 of i] {};
	\path[->]
		(i) edge[thick, bend left] (t);
	\node[state] (1') [right= 0.9 of t] {};
	\node[] (21') [above right = 0.7 and 0.3 of 1'] {};
	\node[] (22') [above right = 0.3 and 0.7 of 1'] {};
	\node[] (31') [below right = 0.3 and 0.7 of 1'] {};
	\node[] (32') [below right = 0.7 and 0.3 of 1'] {};
	\path[->]
		(1') edge[loop,rotate=120] node[pos=0.5,above] {$a$} (1')
			 edge[gray] (21')
			 edge[gray] (22')
			 edge[gray] (31')
			 edge[gray] (32');
	\foreach \n [count=\count from 1] in {1,2,3}{
		\node[dot,gray] (2d\n) at ($(1)+(24+\count*10:0.4cm)$) {};}
	\foreach \n [count=\count from 1] in {1,2,3}{
		\node[dot,gray] (2d\n) at ($(3)+(-20+\count*10:0.4cm)$) {};}
	\foreach \n [count=\count from 1] in {1,2,3}{
		\node[dot,gray] (2d\n) at ($(1')+(24+\count*10:0.4cm)$) {};}
	\foreach \n [count=\count from 1] in {1,2,3}{
		\node[dot,gray] (2d\n) at ($(1')+(-63+\count*10:0.4cm)$) {};}
	\draw (1.6,-1.2) node {(b)};
\end{scope}
\begin{scope}[yshift=-3cm]
	\node[state] (1) {};
	\node[state] (2) [right = of 1] {};
  		\node[] (21) [above right = 0.3 and 0.7 of 2] {};
  		\node[] (22) [below right = 0.3 and 0.7 of 2] {};
	\path[->]
        (1) edge[bend left] node[pos=0.5,above] {$a$} (2)
			edge[bend right] node[pos=0.5,below] {$a$} (2)
        (2) edge[gray] (21)
            edge[gray] (22);
    \node[] (i) [right = 2.3 of 1] {};
    \node[] (t) [right = 0.9 of i] {};
    \path[->]
        (i) edge[thick, bend left] (t);
  \node[state] (1') [right= 0.3 of t] {};
  \node[state] (2') [right =  1.3 of 1'] {};
  	\node[] (21') [above right = 0.3 and 0.7 of 2'] {};
  	\node[] (22') [below right = 0.3 and 0.7 of 2'] {};
  \path[->]
        (1') edge node[pos=0.5,above] {$a$} (2')
        (2') edge[gray] (21')
             edge[gray] (22');
	\foreach \n [count=\count from 1] in {1,2,3}{
      \node[dot,gray] (2d\n) at ($(2)+(-20+\count*10:0.4cm)$) {};}
	\foreach \n [count=\count from 1] in {1,2,3}{
      \node[dot,gray] (2d\n) at ($(2')+(-20+\count*10:0.4cm)$) {};}
	\draw (2.5,-0.9) node {(c)};
\end{scope}
\begin{scope}[yshift=-3cm, xshift=8.5cm]
	\node[state] (1) {};
		\node[] (11) [above right = 0.3 and 0.7 of 1] {};
		\node[] (12) [below right = 0.3 and 0.7 of 1] {};
	\path[->]
		(1) edge[loop,rotate=120] node[pos=0.5,above] {$a$} (1)
			edge[loop,rotate=180] node[pos=0.5,below] {$a$} (1)
			edge[gray] (11)
			edge[gray] (12);
	\node[] (i) [right = 0.8 of 1]{};
	\node[] (t) [right = 0.9 of i] {};
	\path[->]
		(i) edge[thick, bend left] (t);
	\node[state] (1') [right= 1.0 of t] {};
		\node[] (11') [above right = 0.3 and 0.7 of 1'] {};
		\node[] (12') [below right = 0.3 and 0.7 of 1'] {};
	\path[->]
		(1') edge[loop, rotate=150] node[pos=0.7,above] {$a$} (1')
			edge[gray] (11')
			edge[gray] (12');
	\foreach \n [count=\count from 1] in {1,2,3}{
		\node[dot,gray] (2d\n) at ($(1)+(-20+\count*10:0.4cm)$) {};}
	\foreach \n [count=\count from 1] in {1,2,3}{
		\node[dot,gray] (2d\n) at ($(1')+(-20+\count*10:0.4cm)$) {};}
	\draw (1.1,-0.9) node {(d)};
\end{scope}
\end{tikzpicture}
\vspace{-10pt}
\caption{(a), (b) open (or homotopic) foldings; (c), (d) closed (or non-homotopic) foldings}
\label{fig:foldings}
\end{figure}

Stallings bijection behaves well with respect to inclusions in the sense that, for any two subgroups $H,K\leqslant F(A)$, the inclusion $H\leqslant K$ holds if and only if there is a (unique) $A$-homomorphism $\theta_{H,K}\colon \Gamma_A(H) \to \Gamma_A(K)$ (sending vertices to vertices, $\bp_H$ to $\bp_K$, and $a$-labelled edges to $a$-labelled edges, for each $a\in A$).
When $H\leqslant K\leqslant F(A)$ and $\Gamma_A(H)$ is a subautomaton of $\Gamma_A(K)$, that is, the map $\theta_{H,K}$ is injective, then $H$ is a free factor of $K$ (take a spanning tree of $\Gamma_A(H)$ and extend its image under $\theta_{H,K}$ to a spanning tree for $\Gamma_A(K)$, resulting in extending the corresponding basis for $H$ to a basis for $K$), whereas the converse is far from true. Dually, the case where $\theta_{H,K}$ is \emph{onto}, at the level of edges, will be of central interest for our purposes. Here, we are particularly interested in the case $K=\gen{H,g}$.

Let us finish with some technical details, which will be needed later. Given a finite $A$-automaton $\Gamma$, let us distinguish between the \emph{language} $\LL(\Gamma)\leqslant F(A)$ of $\Gamma$,  and its \emph{fundamental group} at the basepoint $\pi(\Gamma, \bp)$. Of course, both are free groups and ``reading the label" gives us an epimorphism $\ell\colon \pi(\Gamma, \bp)\twoheadrightarrow \LL(\Gamma)$, $\gamma \mapsto \ell(\gamma)$; in particular, $\rk(\pi(\Gamma, \bp))=b_1(\Gamma)=1-|V\Gamma|+|E\Gamma|\geqslant \rk(\LL(\Gamma))$. Suppose now that $\Gamma \curvearrowright \Gamma'$ is an elementary folding, identifying the edges $e_1, e_2\in E\Gamma$ with $\iota e_1=\iota e_2=p$, $\tau e_1=q_1$, $\tau e_2 =q_2$, and $\ell(e_1)=\ell(e_2)=a\in A$. The natural map $\varphi\colon \Gamma \to \Gamma'$ induces an epimorphism $\varphi\colon \pi(\Gamma, \bp)\twoheadrightarrow \pi(\Gamma',\bp)$, such that $\varphi \ell=\ell$, and $\varphi$ is an isomorphism if and only if the folding is open (i.e., $q_1\neq q_2$). Given a reduced path $\gamma$ in $\Gamma$, the action of $\varphi$ \emph{projects} $\gamma$ in a natural way to a path $\gamma'$ in $\Gamma'$ (which may not be reduced, e.g. when $\gamma$ goes through $e_1$ and immediately returns through $e_2$).

In the other direction, when $\gamma'$ is a reduced path in $\Gamma'$, then we can \emph{lift} $\gamma'$ to a path $\gamma$ in $\Gamma$, but the lifting is not necessarily unique.
When the folding $\Gamma \curvearrowright \Gamma'$ is closed, then the lifting of $\gamma'$ can simply be done in the following way: whenever we pass (in either direction) through the identified ($e_1=e_2$)-edge, we choose between $e_1$ and $e_2$ to be in $\gamma$, while all the vertices and occurrences of other edges along the path stay the same.
When the folding $\Gamma \curvearrowright \Gamma'$ is open then it can happen that the path $\gamma'$ goes through the identified vertex $q_1=q_2$ in such a way that, after the unfolding and separation of these vertices, it needs to enter the vertex $q_1$ and immediately exit the vertex $q_2$ (or the other way round). The remedy to this cut in the path is to connect $q_1$ with $q_2$ by inserting the segment $e_1^{-1}e_2$ (or $e_2^{-1}e_1$ in the reverse direction).
At the level of labelling, $\ell(\gamma)$ may thus not be reduced, even when $\ell(\gamma')$ is reduced; however, both labels represent the same element in the free group $F(A)$).

Since 1983, Stallings bijection became a central tool for the modern understanding and study of the lattice of subgroups of a free group. With the development of these graph-theoretical techniques, many new results have been obtained about free groups and their subgroups. Also, most of the results known before Stallings~\cite{Sta83} have been reproved using graph-theory techniques, usually with conceptually simpler and more transparent proofs.

\section{An explicit description of the set of dependent elements} \label{sec:all_dependent_elements}

Let $|A|<\infty$ be a finite alphabet and let $H\leqfg F(A)$ and $1\neq g\in F(A)$ be given. To check whether $g$ is dependent on $H$, we just have to check whether $\rk(\gen{H,g})\leqslant \rk(H)$. One way to do this is by using classical Nielsen transformations; see~\cite{LS01}. Another way is within Stallings graph theory setting. We start by computing $\Gamma_A(H)$ and then we form the wedge of $\Gamma_A(H)$ with a petal spelling the (reduced) word $g$. Let $\Gamma_0$ be the resulting $A$-automaton. Clearly, $\LL(\Gamma_0)=\gen{H,g}$. It is not necessarily deterministic but, after performing a maximal sequence of foldings to $\Gamma_0$, we obtain the Stallings automaton $\Gamma_A(\gen{H,g})$:
 \begin{equation}\label{eq:cascade}
\Gamma_0 \curvearrowright \Gamma_1 \curvearrowright \Gamma_2 \curvearrowright \cdots \curvearrowright \Gamma_s=\Gamma_A(\gen{H,g})
 \end{equation}
(again, independent of the order in which we perform the foldings). Counting the first Betti numbers for these graphs, we have $b_1(\Gamma_0)=b_1(\Gamma_A(H))+1=\rk(H)+1$ and $b_1(\Gamma_A(\gen{H,g}))= \rk(\gen{H,g})$. Since at each elementary folding $\Gamma_i \curvearrowright \Gamma_{i+1}$ the first Betti number either remains unchanged or decreases by one, we deduce that
 $$
\begin{array}{rcl} g \text{ is dependent on } H & \Leftrightarrow & \rk(\gen{H,g})\leqslant \rk(H) \\ & \Leftrightarrow & b_1(\Gamma_A(\gen{H,g}))\leqslant b_1(\Gamma_0)-1 \\ & \Leftrightarrow & \eqref{eq:cascade} \text{ contains at least one closed folding.}
\end{array}
 $$

Let us look in more detail into the process of obtaining $\Gamma_A(\gen{H,g})$ from $\Gamma_0$. Since the result does not depend on the order of foldings, we can choose an order that suits our needs. We start at the basepoint and fold one by one the edges of the $g$-petal into those of the deterministic $A$-automaton $\Gamma_A(H)$, as long as this is possible. Let $g=g_+ g'$, with $g_+$ being the maximal initial segment of $g$ that we managed to fold (where $g_+$ or $g'$ may be trivial), and let $p$ be the vertex in $\Gamma_A(H)$ at which the path spelling $g_+$ ends. Note, that, by maximality of $|g_+|$, if $g'$ is non-trivial then its first letter does not equal any of the labels of the edges in $\Gamma_A(H)$ leaving $p$. In a similar way, we fold, edge by edge and backwards, into $\Gamma_A(H)$ the longest possible terminal segment of the path spelling $g'$. Let $g'=g_0 g_-^{-1}$ be the corresponding decomposition of $g'$, with $g_-$ spelling the label of a path from $\bp$ to a vertex, say $q$, in $\Gamma_A(H)$. Again, by maximality of $|g_-|$, if $g_0$ is non-trivial then its last letter does not equal any of the labels of the edges in $\Gamma_A(H)$ terminating at $q$; see Figure~\ref{fig:cascade}. Note, that all the elementary foldings performed until now were open. At this stage we have the decomposition $g=g_+ g_0 g_-^{-1}$, with no cancellation among the three factors. Let us distinguish between the following three cases:
\begin{enumerate}
\item \emph{$g_0\neq 1$}. In this case the folding process is complete and the obtained $A$-automaton $\Gamma_A(\gen{H,g})$ consists of $\Gamma_A(H)$ with an arc of positive length spelling $g_0$ attached to it and going from $p$ to $q$; see Figure~\ref{fig:cascade}(b). In this case, $b_1(\Gamma_A(\gen{H,g}))= b_1(\Gamma_A(H))+1=b_1(\Gamma_0)$ and so, $g$ does not depend on $H$.
\item \emph{$g_0=1$ and $p=q$}. In this case, $g=g_+g_-^{-1}$ is the label of a genuine $\bp$-cycle in $\Gamma_A(H)$. Thus, $g\in H$ and $\gen{H,g}=H$. Clearly, $g$ is dependent on $H$ in this case.
\item \emph{$g_0=1$ and $p\neq q$}. Here, as in case (i), $g\not\in H$, but at this stage it is not clear yet whether $g$ is dependent on $H$ or not. At this moment, we have exhausted the original petal $g$ and have identified the vertices $p$ and $q$, obtaining the automaton $\Gamma_A(H)/(p=q)$. This automaton is not necessarily deterministic and further foldings may need to be performed until reaching $\Gamma_A(\gen{H,g})$. In this case, $g$ is dependent on $H$ if and only if at least one closed folding occurs in this second part of the process.
\end{enumerate}

\begin{figure}
\centering
\begin{tikzpicture}[shorten >=1pt, node distance=1.2 and 1.2, on grid,auto,>=stealth',decoration={snake, segment length=2mm, amplitude=0.5mm,post length=1.5mm}]
\begin{scope}[scale=0.8]
\draw [gray, fill=gray!20] plot [smooth cycle] coordinates {(-1,0) (1,1) (3,1) (5,1) (4,0) (2,-1)};
\node[state, accepting] (0) {};
\node[] (lp) [above right = 0.7 and 3.4 of 0] {$\scriptstyle{\Gamma_A(H)}$};
\node[state] (p) [above right = 0.5 and 1.5 of 0] {};
\node[state] (q) [below right = 0.25 and 2.5 of 0] {};
\node[] (lp) [above left = 0.15 and 0.15 of p] {$\scriptstyle{q}$};
\node[state] (r) [below left = 0.7 and 1.5 of 0] {};
\node[state] (s) [above left = 0.7 and 1.5 of 0] {};
\path[->] (0) edge[snake it, bend right, min distance=7mm, in=70, out=70, blue] node[below=0.2] {$g_{+}$} (r);
\path[->] (0) edge[snake it, bend right, min distance=7mm, in=-70, out=-70, blue] node[above=0.2] {$g_{-}$} (s);
\path[->] (r) edge[snake it, bend left, min distance=15mm, in=70, out=70, blue] node[above=0.5] {$g_0$} (s);
\path[->] (0) edge[snake it] node[above left] {$g_{-}$} (p);
\path[->] (0) edge[snake it] node[below] {$g_{+}$} (q);
\node[] (lp) [below right = 0.1 and 0.2 of q] {$\scriptstyle{p}$};
\node[] [below right =1.5 and 1 of 0]{(a)};
\end{scope}
\begin{scope}[scale=0.8, shift={(7,0)}]
\draw [gray, fill=gray!20] plot [smooth cycle] coordinates {(-1,0) (1,1) (3,1) (5,1) (4,0) (2,-1)};
\node[state, accepting] (0) {};
\node[] (lp) [above right = 0.7 and 3.4 of 0] {$\scriptstyle{\Gamma_A(H)}$};
\node[state] (p) [above right = 0.5 and 1.5 of 0] {};
\node[state] (q) [below right = 0.25 and 2.5 of 0] {};
\node[] (lp) [above left = 0.15 and 0.15 of p] {$\scriptstyle{q}$};
\path[->] (0) edge[snake it] node[above left] {$g_{-}$} (p);
\path[->] (0) edge[snake it] node[below] {$g_{+}$} (q);
\node[] (lp) [below right = 0.1 and 0.2 of q] {$\scriptstyle{p}$};
\path[->] (q) edge[snake it, bend left, min distance=30mm, in=-70, out=-70, blue] node[above=0.2] {$g_0\neq 1$} (p);
\node[] [below right =1.5 and 1 of 0]{(b)};
\end{scope}
\end{tikzpicture}
\caption{The $A$-automata (a) $\Gamma_0$; and (b) $\Gamma(\gen{H,g})$ in case~(i)}
\label{fig:cascade}
\end{figure}
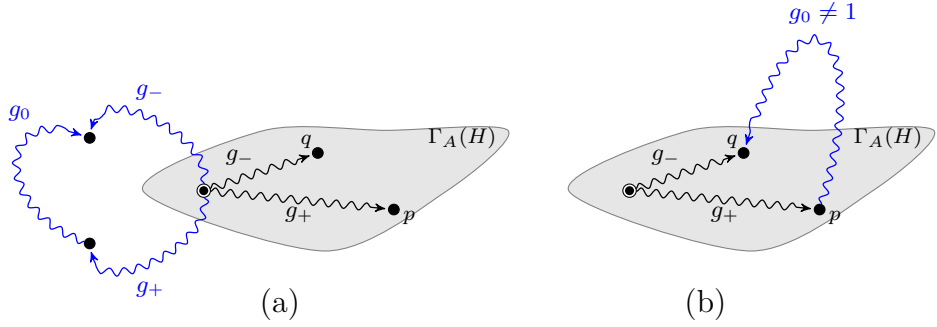

Summarizing, $g$ depends on $H$ if and only if $g\in H$, or $g$ goes into case (iii) with at least a closed folding occurring in the second part of the process. The above discussion leads us to an explicit description of \emph{all} elements dependent on $H$.

\begin{theorem}\label{thm:depgen}
Let $F(A)$ be the free group with basis $A$. Given a finite set of generators of total length $n$ of a subgroup $H\leqfg F(A)$, one can algorithmically compute in time $\OO(n^3\log^*(n))$ a finite set $\{g_0=1, g_1,\ldots ,g_m\}$ of elements of $F(A)$ that depend on $H$, and such that $\dep(H)$ is the disjoint union of the corresponding double cosets, $\dep(H)=H\sqcup Hg_1H\sqcup \cdots \sqcup Hg_mH$.
\end{theorem}

\begin{proof}
Since $H$ is finitely generated, there exists a finite subset $A_0\subseteq A$ (computable in linear time), such that $H\leqfg F(A_0)\leqff F(A)$; moreover, $\dep_{F(A)}(H)=\dep_{F(A_0)}(H)$. Thus, without loss of generality, we may assume that $A$ is finite, say $A=\{a_1, \ldots, a_r\}$.

First, we construct the finite Stallings $A$-automaton $\Gamma_A(H)$. Then, for each pair of distinct vertices $p,q\in V=V\Gamma_A(H)$, we form the graph $\Gamma_A(H)/(p=q)$ by identifying these two vertices, and then we perform elementary foldings until reaching a deterministic $A$-automaton, say $\Gamma_{p=q}$. Let
 \begin{equation}\label{eq:P}
R=\big\{ (p,q) \in V\times V \mid p\neq q,\,\, b_1(\Gamma_{p=q})\leqslant \rk(H)\big\} \subseteq V\times V\setminus \Delta,
 \end{equation}
where $\Delta=\{(p,p) \mid p\in V\}$ is the diagonal of $V\times V$. In other words, $R$ is the set of pairs $(p,q)\in V\times V\setminus \Delta$, for which the sequence of elementary foldings towards $\Gamma_{p=q}$ contains at least one which is closed. Note, that $(p,q)\in R$ implies that $(q,p)\in R$ and also $(pa_i^{\varepsilon}, qa_i^{\varepsilon})\in R$, for all $a_i\in A$ and $\varepsilon \in \{1,-1\}$, for which both $pa_i^{\varepsilon}$ and $qa_i^{\varepsilon}$ are defined.

Clearly, the discussion above tells us that
 \begin{equation}\label{eq:pairs}
\dep(H)=H\sqcup \bigcup_{(p,q)\in R} S_{(p,q)},
 \end{equation}
where $S_{(p,q)}$ is the set of elements $g \in F(A)$ that fall into case (iii), that is,
 $$
S_{(p,q)}=\{g=g_+g_-^{-1} \mid g_+ = \ell(\gamma_p), g_- = \ell(\gamma_q), \iota\gamma_p= \iota\gamma_q= \bp,\,\, \tau\gamma_p=p,\,\, \tau\gamma_q=q\}.
 $$
It only remains to observe that each of these sets $S_{(p,q)}$ is the double coset $S_{(p,q)}= Hg_{(p,q)}H$, for some (any) representative $g_{(p,q)}\in S_{(p,q)}$. Since $H$ is finitely generated, $\Gamma_A(H)$ and $R$ are finite, and so, the union in~\eqref{eq:pairs} is also finite. In addition, the coset representatives $g_{(p,q)}$ are computable. Moreover, we can algorithmically check whether $HuH=HvH$, i.e., $uH\cap Hv\neq \emptyset$, by using a solution to the coset intersection problem for free groups; see~\cite{DV21}. So, by writing $g_{(\bp,\bp)}=1$ and $S_{(\bp,\bp)}=H$, and keeping in the set $\{g_{(p,q)} \mid (p,q)\in R\}$ only one representative of each double coset, we obtain the required elements $g_0=1, g_1, \ldots, g_m$, such that:
 $$
\dep(H)=\bigcup_{(p,q)\in R\sqcup \{(\bp,\bp)\}} S_{(p,q)}=H\sqcup Hg_1H\sqcup \cdots \sqcup Hg_mH.
 $$

If $n$ denotes the total length of the input generators, $\Gamma_A(H)$ can be computed in time $\OO(n\log^*(n))$; see Touikan~\cite{Tou06}. This $A$-automaton contains at most $n$ vertices, and so, $\OO(n^2)$ pairs $(p,q)$. For each such pair $(p,q)$, we can check in time $\OO(n\log^*(n))$ whether, after identifying $p$ with $q$, the elementary foldings towards $\Gamma_{p=q}$ contain a closed folding. This makes a total time $\OO(n^3\log^*(n))$ to compute the set $R$.

Finally, the process of removing from $R$ the excessive pairs of vertices, so that no two pairs determine the same double coset, can be done in time $\OO(n^3)$. In order to do it, we first compute a spanning tree $T$ of $\Gamma_A(H)$ (in time $\OO(n)$). Then, for each $p\in V\Gamma_A(H)$, we attach to $\Gamma_A(H)$ a ``hair", labelled $\ell(T[\bp, p])$, that ends at the basepoint $\bp$. After making the necessary elementary foldings, we obtain a ``hairy" $A$-automaton, denoted $\Gamma'$, with at most $\OO(n^2)$ vertices. Next, we compute the pull-back of $\Gamma'$ with $\Gamma_A(H)$, which takes time $\OO(n^3)$ (see~\cite{DV21}). For each pair of vertices $p,q\in V\Gamma_A(H)$, let $u=\ell(T[\bp,p])$ and $v=\ell(T[\bp,q])$. Then, $HuH=HvH$ if and only if $uH\cap Hv\neq \emptyset$ if and only if $(\bp u^{-1},\bp)$ and $(\bp,q)$ belong to the same connected component of the pull-back.

Altogether, the whole process can be done in time $\OO(n^3\log^*(n))$.
\end{proof}

\begin{remark}
The algorithm for computing the representatives $g_i$ of the double cosets whose union is $\dep(H)$ can also be done using Nielsen transformations. Suppose that $u=\ell(T[\bp,p])$ and $v=\ell(T[\bp,q])$, as in the proof of Theorem~\ref{thm:depgen}. Let $S=W \cup \{g\}$, where $W$ is a basis of $H$ and $g=u v^{-1}$. Then, an occurrence of a closed elementary folding, after the identification of $p$ with $q$, is equivalent to checking whether performing a sequence of Nielsen transformations to $S$ leads to one of the generators being reduced to the trivial group element.

In fact, an identification of two edges of the same initial vertex and the same label $a$ is equivalent to the reduction of a word in which the letter $a^{\varepsilon}$ is followed by $a^{-\varepsilon}$.
	
Similar to performing Nielsen transformations but using a different representation is the following. We can represent each group element $g=u v^{-1}$ as a \emph{binomial} $u-v$ in the group ring $\Z_2 F$ of the free group $F$ over the ring of integers modulo 2. Then, the analog of a Nielsen transformation is a reduction \`{a} la Gr\"{o}bner bases, using, say, a \emph{shortlex} order (comparing monomials by their length and using a lexicographic order for a tie-breaker); see \cite[Sec. 7]{Ros93}. The advantage of this method is that it allows, in some cases, the generalization of theorems and algorithms from groups to group rings. On the other hand, Stallings graph-theory techniques, besides the new perspective and results they brought to the field of subgroups of free groups, show, in many cases, a simplified and unified approach: for example, the graph representation of the subgroup $H \leqslant F$ does not depend on the generators of $H$ but only on the chosen basis for $F$.
\end{remark}

\begin{corollary}\label{cor:basis}
Given a finite set of generators of total length $n$ for a subgroup $H\leqfg F(A)$, a generating set for $\Dep(H)$ can be algorithmically computed in time $\OO(n^3\log^*(n))$.
\end{corollary}

\begin{proof}
Clearly, if $\{h_1, \ldots ,h_s\}$ are the given generators for $H$, then $\Dep(H)=\gen{h_1, \ldots ,h_s, g_1, \ldots ,g_m}$, where $g_1, \ldots ,g_m$ are the elements computed in Theorem~\ref{thm:depgen}. The result then follows.
\end{proof}

\section{Equations for a dependent element}\label{sec:equations}

Let $H\leqfg F(A)$ and let $g\in F(A)$ be an element dependent on $H$. In this section we present two algorithms for computing a finite set of polynomials that generate $I_H(g)$ (the ideal of all polynomials $w(x)$, such that $w(g)=1$) as a normal subgroup of $H*\gen{x}$. The first one is based on Nielsen transformations and the second one on Stallings graph theory techniques. Let us start with the easy case $g\in H$.

\begin{observation}\label{obs:trivialcase}
For $g\in H\leqfg F(A)$, we have $I_H(g)=\, \ll g^{-1}x \gg \unlhd H*\gen{x}$.
\end{observation}

\begin{proof}
Clearly, $g^{-1}x \in I_H(g)$. Let now $w(x)=1$ be a reduced $H$-equation satisfying $w(g)=1$. We will show, by induction on the degree $d$ of $w(x)$, that $w(x)$ is in the normal subgroup $\ll g^{-1}x \gg \, = \, \ll gx^{-1} \gg$.	For $d=1$, we have $w(x)=h_0x^{\varepsilon_1} h_1$ and $h_0g^{\varepsilon_1}h_1=1$. So, $g^{-\varepsilon_1}=h_1h_0$ and
 $$
w(x)=h_0x^{\varepsilon_1}h_1=h_1^{-1}(h_1h_0x^{\varepsilon_1})h_1 =h_1^{-1}(g^{-\varepsilon_1}x^{\varepsilon_1})h_1\in \, \ll g^{-1}x \gg.
 $$
Assume now that the statement holds for degree $d-1$ and let
 $$
w(x)=h_0 x^{\varepsilon_1}h_1x^{\varepsilon_2}\cdots h_{d-1}x^{\varepsilon_d}h_d=1
 $$
be a reduced $H$-equation of degree $d$ and satisfying $w(g)=1$. Then, $g$ clearly satisfies the $H$-equation
 $$
w'(x) = h_0 x^{\varepsilon_1}h_1x^{\varepsilon_2}\cdots x^{\varepsilon_{d-1}}(h_{d-1}g^{\varepsilon_d}h_d)=1
 $$
of degree $d-1$, and so, by the induction hypothesis, $w'(x) \in \, \ll g^{-1}x \gg$. It follows that
 $$
w(x)=(h_0 x^{\varepsilon_1}h_1x^{\varepsilon_2}\cdots x^{\varepsilon_{d-1}}h_{d-1} g^{\varepsilon_d}h_d) (h_d^{-1}g^{-\varepsilon_d} x^{\varepsilon_d}h_d)=w'(x)(g^{-\varepsilon_d}x^{\varepsilon_d})^{h_d},
 $$
which also belongs to $\ll g^{-1}x \gg$.
\end{proof}

\begin{theorem}\label{thm:eqbasis}
Let $H\leqfg F(A)$ be of rank $n$ and let $1\neq g\in F(A)$ be dependent on $H$, such that $\rk(\gen{H,g}) = m \leqslant n$. Then, there exists an algorithm that, given a finite set of generators for $H$, computes a set of $n+1-m$ normal generators of the ideal $I_H(g) \unlhd H*\gen{x}$:
 $$
I_H(g) =\, \ll w_1(x), \ldots ,w_{n-m+1}(x)\gg.
 $$
\end{theorem}

\begin{proof}[Proof (using Nielsen transformations)]
Let $\{h_1, \ldots, h_n\}$ be a set of free generators of $H$. As in \eqref{eq:phi}, we consider the epimorphism
 $$
\begin{array}{rcl}
\varphi_{H,g} \colon H*\gen{x} & \twoheadrightarrow & \gen{H,g}\leqslant F(A) \\ h_1 & \mapsto & h_1, \\ & \vdots & \\ h_n & \mapsto & h_n, \\ x & \mapsto & g.
\end{array}
 $$
By applying a sequence of Nielsen transformations, the tuple $(h_1, \ldots, h_n, g)$ can be transformed into the tuple $(1,\ldots,1,b_1, \ldots ,b_m)$, consisting of $n-m+1$ trivial elements, followed by a basis $\{b_1, \ldots ,b_m\}$ of $\gen{H,g}$:
 $$
(h_1, \ldots, h_n, g) \looparrowright \cdots \looparrowright (1,\ldots, 1, b_1, \ldots ,b_m).
 $$
Applying the same sequence of transformations to the tuple $(h_1, \ldots, h_n, x$) results in a new basis $\{w_1(x), \ldots, w_{n+1}(x)\}$ of the free group $H*\gen{x}$. The epimorphism $\varphi_{H,g}$, expressed in this new basis, is
 $$
\begin{array}{rcl}
\varphi_{H,g} \colon H*\gen{x} & \twoheadrightarrow & \gen{H,g} \\ w_1(x) & \mapsto & 1, \\ & \vdots & \\ w_{n-m+1}(x) & \mapsto & 1, \\ w_{n-m+2}(x) & \mapsto & b_1, \\ & \vdots & \\ w_{n+1}(x), & \mapsto & b_{m}. \end{array}
 $$
Since $\{b_1, \ldots ,b_m\}$ is a basis of $\gen{H,g}$, the restriction of $\varphi_{H,g}$ to $\gen{w_{n-m+2}(x), \ldots$, $w_{n+1}(x)}$ is injective, thus the kernel of $\varphi_{H,g}$ is
 $$
\ker \varphi_{H,g} = \, \ll w_{1}(x),\ldots, w_{n-m+1}(x)\gg,
 $$
the normal subgroup of $H*\gen{x}$, normally generated by $w_{1}(x),\ldots,w_{n-m+1}(x)$.
\end{proof}

\begin{proof}[Proof (using Stallings techniques)]
An alternative way to obtain a set of normal generators for the ideal $I_H(g)$ is through Stallings techniques. The basic case $g\in H$ was already covered in Observation~\ref{obs:trivialcase}: in this case $\gen{H,g}=H$, $m=n$, $n-m+1=1$, and $I_H(g)$ is normally generated by $g^{-1}x$.

So, let us assume that $g$ is dependent on $H$ and $g\notin H$, i.e., case (iii) in the above discussion. Let $W=\{h_1, \ldots ,h_n\}$ be a free basis for $H$ and let $\FF(W\cup \{g\})$ be the corresponding flower automaton. We will now perform a series of elementary foldings towards the deterministic $A$-automaton $\Gamma_A(\gen{H,g})$, in an order that suits our needs. At the first stage we only fold the elements of $W$ and leave the petal $g$ as it is, obtaining the automaton $\Gamma_0$, as in ~\eqref{eq:cascade}, which is the wedge of $\Gamma_A(H)$ with the petal labelled $g$:
 $$
\FF(W\cup \{g\})\curvearrowright \cdots \curvearrowright \Gamma_0.
 $$
Note that, up to now, all the foldings are open because $W$ is a free basis for $H$. In a second stage, we integrate $g$ in $\Gamma_A(\gen{H,g})$ through a series of elementary foldings as in~\eqref{eq:cascade}, but doing \emph{only open foldings}, until no more possible open foldings are available: first we fold the maximal initial and terminal segments of the petal labelled $g$ until getting the automaton $\Gamma_A(H)/p=q$, and then we keep applying \emph{open} foldings until no more are available; denote $\Gamma_{p,q}$ the resulting automaton:
 $$
\Gamma_0 \curvearrowright \cdots \curvearrowright \Gamma_A(H)/(p=q) \curvearrowright \cdots \curvearrowright \Gamma_{p,q}.
 $$
The automaton $\Gamma_{p,q}$ is not yet necessarily deterministic: it could contain several closed foldings to be done. But note that, after doing all of them, no more identifications of vertices are happening, no new foldings are introduced and so, we get directly the final deterministic output
 $$
\Gamma_{p,q} \,\widehat{\,\curvearrowright}\,\, \Gamma_A(\gen{H,g}).
 $$
Let us introduce notation for this final step: $\Gamma_{p,q}$ contains, say $r\geqslant 0$, blocks of edges (possibly loops) parallel to each other, namely $\{ e_{i,0}, e_{i,1},\ldots ,e_{i,k_i}\}$ for $i=1,\ldots ,r$, with $\iota e_{i,0}=\iota e_{i,1}=\cdots =\iota e_{i,k_i}$, and $\tau e_{i,0}=\tau e_{i,1}=\cdots \tau e_{i,k_i}$, and $\ell(e_{i,0})=\ell(e_{i,1})=\cdots =\ell(e_{i,k_i})$. This final step, denoted $\widehat{\,\curvearrowright}$, aggregates
 $$
\sum_{i=1}^{r}k_i =n-m+1
 $$
closed foldings, and consists simply on identifying each block into a single edge; alternatively, we can look at it as deleting the edges $e_{i,1},\ldots ,e_{i,k_i}$ and leaving only $e_{i,0}$ from each block. The whole process of folding looks like this
 \begin{equation}\label{eq:cascade2}
\FF(W\cup \{g\})\curvearrowright \cdots \curvearrowright \Gamma_0 \curvearrowright \cdots \curvearrowright \Gamma_A(H)/(p=q) \curvearrowright \cdots \curvearrowright \Gamma_{p,q} \widehat{\,\curvearrowright}\, \Gamma_A(\gen{H,g}),
 \end{equation}

	
Next, we will show how to construct $n-m+1$ non-trivial $H$-equations $w_{i,j}(x)=1$ satisfied by $g$ for each of the edges $e_{i,j}$, $i=1,\ldots,r$, $j=1,\ldots ,k_i$. Then, we will prove that the polynomials $w_{i,j}(x)$ generate the ideal $I_H(g)$ as a normal subgroup of $H*\gen{x}$.

For each $1 \leqslant i \leqslant r$, let $\gamma_{i}$ be a reduced path in $\Gamma_{p,q}$ from $\bp$ to $\iota e_{i,0}=\cdots =\iota e_{i,k_i}$, which does not contain any of the edges $e_{i,0},\ldots,e_{i,k_i}$ (or their inverses) (if this is not possible, we define $e_{i,0},\ldots,e_{i,k_i}$ to be in the reverse direction, and then take a path $\gamma_{i}$ as above - this is always possible since every vertex lies on a reduced $\bp$-cycle). For each $i = 1,\ldots, r$ and for each $j=1,\ldots,k_i$, let now $\gamma_{i,j}$ be the reduced non-trivial $\bp$-cycle
 $$
\gamma_{i,j}=\gamma_i e_{i,j} e_{i,0}^{-1} \gamma_i^{-1}.
 $$
Let $a^{\varepsilon}$ be the label of $e_{i,j}$. Then, the label of $\gamma_{i,j}$ is
 $$
\ell(\gamma_{i,j})= \ell(\gamma_i)a^{\varepsilon}a^{-\varepsilon} \ell(\gamma_i)^{-1}=1;
 $$
this is capturing the algebraic essence of the closed folding given by the pair of edges $e_{i,0}$ and $e_{i,j}$. In the next stage, we lift $\gamma_{i,j}$ all the way up to $\FF(W\cup \{g\})$ by unfolding the open foldings in a reverse order of \eqref{eq:cascade2} to obtain a $\bp$-cycle $\widehat{\gamma}_{i,j}$.
The $\bp$-cycle $\widehat{\gamma}_{i,j}$ in $\FF(W\cup \{g\})$ is non-trivial and (topologically) reduced. So, it spells a non-trivial, reduced word, denoted $w_{i,j}(g)$, on $W\cup \{g\}$, but reads the trivial element $\ell(\widehat{\gamma}_{i,j})= \ell(\gamma_{i,j})=1$. Replacing the occurrences of $g$ by $x$ (and of $g^{-1}$ by $x^{-1}$) in $w_{i,j}(g)$, results in a reduced, non-trivial polynomial $w_{i,j}(x)$, satisfying $w_{i,j}(g)=1$ (note, that $x$ must occur in $w_{i,j}(x)$ at least once, because $W$ is a free basis of $H$).

Let us rename the $n-m+1$ polynomials $w_{i,j}(x)$ obtained this way by $w_1(x),\ldots ,$ $w_{n-m+1}(x)$. We finish the proof by showing that they generate normally the ideal $I_H(g)$. So, let $w(x)\in I_H(g)$ be an arbitrary reduced polynomial over $H$, satisfying $w(g)=1$. Then, $w(x)$ corresponds to a reduced $\bp$-cycle $\gamma$ in the flower $A$-automaton $\FF(W\cup \{g\})$. Note, that the condition $w(g)=1$ means that $\ell(\gamma)=1$. We project $\gamma$ down the open foldings in the cascade~\eqref{eq:cascade2} until getting a reduced $\bp$-cycle $\gamma'$ in $\Gamma_{p,q}$, satisfying $\ell(\gamma')=1$. We are going to prove our goal $w(x)\in \, \ll w_1(x),\ldots ,w_{n-m+1}(x)\gg$ by induction on the number of occurrences, say $s$, of the duplicated edges $e_{i,j}$, $i=1,\ldots,r$, $j=1,\ldots,k_i$, in $\gamma'$. If $s=0$ then the same path $\gamma'$ is, in fact, also a reduced $\bp$-cycle in $\Gamma_A(\gen{H,g})$ because the closed foldings do not affect it. But $\Gamma_A(\gen{H,g})$ is a deterministic $A$-automaton and therefore $\gamma'$ is the trivial path. We claim that this implies that $\gamma$ is also the trivial path in $\FF(W\cup \{g\})$. The reason is that all the foldings on the way from $\FF(W\cup \{g\})$ to $\Gamma_{p,q}$ are open, and after an unfolding of an open folding, a trivial path consisting of a single point remains a single point (which may not be the case when unfolding and reducing a closed folding). It follows that $w(x) \, \in \, \ll~w_1(x),\ldots ,w_{n-m+1}(x)~\gg$. This proves the base of the induction.

Suppose now that the result holds for equations determining paths with $s$ occurrences of the duplicated edges $e_{i,j}$ and let $w(x)\in I_g(H)$ be an equation determining a path $\gamma$ with $s+1$ such occurrences. Then, $\gamma$ can be written as $\gamma=\alpha e_{i,j}^{\varepsilon} \beta$, with $j\geqslant 1$ and $\alpha$ containing no duplicated edge $e_{i,j}$, $i=1,\ldots,r$, $j=1,\ldots,k_i$ (in either direction). Next, we decompose $\gamma$ in the following form:
 $$
\gamma =\left\{ \begin{array}{ll} \alpha e_{i,j}\beta =(\alpha \gamma_i^{-1})(\gamma_ie_{i,j}e_{i,0}^{-1}\gamma_i^{-1})(\gamma_i \alpha^{-1})(\alpha e_{i,0}\beta), & \text{ if } \varepsilon=1; \\  \alpha e_{i,j}^{-1}\beta =(\alpha e_{i,0}^{-1}\beta) (\beta^{-1}\gamma_i^{-1})(\gamma_ie_{i,0}e_{i,j}^{-1}\gamma_i^{-1}) (\gamma_i \beta), & \text{ if } \varepsilon=-1. \end{array} \right.
 $$
Lifting $\gamma$ up to $\FF(W\cup \{g\})$, we get
 $$
\widehat{\gamma}=\left\{ \begin{array}{ll} (\widehat{\gamma_i \alpha^{-1}})^{-1} \widehat{\gamma}_{i,j} (\widehat{\gamma_i \alpha^{-1}}) (\widehat{\alpha e_{i,0}\beta}), & \text{ if } \varepsilon=1; \\ (\widehat{\alpha e_{i,0}^{-1}\beta}) (\widehat{\gamma_i \beta})^{-1} \widehat{\gamma}_{i,j}^{-1} (\widehat{\gamma_i \beta}), & \text{ if } \varepsilon=-1, \end{array} \right.
 $$
which, when translated back into $H$-polynomials, means that $w(x)$ equals the product of a conjugate of one of the polynomials $w_1(x)^{\pm}, \ldots ,w_{n-m+1}(x)^{\pm}$ and the polynomial corresponding to $(\widehat{\alpha e_{i,0}^{\varepsilon}\beta})$, say $w'(x)$. Observe that $w'(g)=1$ as well, and that the corresponding path has $s$ occurrences of duplicated edges (since the first one, $e_{i,j}^{\varepsilon}$, was replaced by $e_{i,0}^{\varepsilon}$). By the induction hypothesis, we have $w'(x)\in \, \ll w_1(x), \ldots ,w_{n-m+1}(x)\gg$ and therefore, so does $w(x)$. This concludes the proof.
\end{proof}

\begin{corollary}
Given $H\leqfg F$, one can effectively compute an integer $d$, such that each $g\in F$ dependent on $H$ satisfies an equation over $H$ of degree at most $d$.
\end{corollary}

\begin{proof}
By Theorem~\ref{thm:depgen}, we can compute elements $g_1, \ldots ,g_m\in F$ dependent on $H$, so that $\dep(H)=H\sqcup Hg_1H\sqcup \cdots \sqcup Hg_mH$. Now, by Theorem~\ref{thm:eqbasis}, we can effectively compute an $H$-equation $w_i(x)=1$ of degree, say, $d_i$ for each $g_i$, $i=1,\ldots ,m$. Since, by Observation~\ref{obs:cosets}, all the elements of the form $hg_i h'\in Hg_iH$ satisfy an $H$-equation of degree less than or equal to $d_i$, it follows that each dependent element satisfies an equation of degree at most $d=\max\{d_1, \ldots ,d_m\}$.
\end{proof}

\begin{remark}
We have seen that when $H\leqfg F$ then the set of elements that depend on $H$ is of the form $\dep(H)=H\sqcup Hg_1H\sqcup \cdots \sqcup Hg_mH$. Let
 $$
J=\bigcup_{i=0}^{m} I_H(g_i) \subseteq H*\gen{x},
 $$
where $g_0 = 1$. Then, the set of $H$-polynomials which are annihilated by some element dependent on $H$ is the union, over the set of all automorphisms of $H*\gen{x}$ that fix $H$, of images of $J$. This is because the $H$-equations satisfied by $hg_ih'$ are, precisely, those of the form $w(h^{-1}xh'^{-1})=1$, for $w(x)\in I_H(g_i)$, and the elements $h^{-1}xh'^{-1}$ are all possible images of $x$ by the automorphisms of $H*\gen{x}$ that fix $H$.
\end{remark}

\section{Compressed subgroups}\label{sec:subgroup_types}

Let us look at some special classes of subgroups of a finitely generated free group $F$ and see whether they are preserved under the dependence operator. These classes form a chain with respect to inclusion, where at the top of the chain are compressed subgroups and at the bottom free factors; see Figure~\ref{fig:subgpschain}.

A subgroup $H\leqslant F$ is called \emph{inert} if $\rk (H\cap K)\leqslant \rk(K)$, for every $K\leqslant F$, and it is called \emph{compressed} when $\rk(H)\leqslant \rk(K)$, for every $H\leqslant K\leqslant F$ (i.e., the same property required only for those subgroups $K$ containing $H$); see~\cite{DV96}. Of course, every inert subgroup of $F$ is compressed, and it is an open problem whether compressed subgroups in free groups are inert. It is also clear that compressed subgroups of $F$ cannot have rank larger than that of the ambient group $F$.

A subgroup $H\leqslant F_r$ is an \emph{echelon} subgroup if there exists an ordered basis $A=\{a_1,\ldots,a_r\}$ of $F_r$, such that $0\leqslant \rk(H_i)-\rk(H_{i-1})\leqslant 1$ for $i=1,\ldots, r$, where $H_0=\{1\}$ and $H_i =H\cap \gen{a_1,\ldots,a_i}$, $i=1,\ldots, r$ (we also say then that $H$ is in \emph{echelon form w.r.t. $A$}); see~\cite{Ros13}. Note, that the inequality $0\leqslant \rk(H_i)-\rk(H_{i-1})$ always holds because $H_{i-1}\leqff H_i$. Clearly, every free factor of $F$ is an echelon subgroup.

The family of echelon subgroups is a subclass of a broader class, called \emph{1-gen-endo}, which was shown by Rosenmann \cite{Ros13} to be inert. A 1-gen-endo subgroup $H$ is formed through a finite number of transformations. We start with a basis $B$ of a free group $F$ and let $H_1$ be the subgroup generated by the elements of $B$ except for some $h \in B$ that is replaced by some element $h' \in F$. Suppose now that at step $i$ we have a subgroup $H_i \leqslant F$. Then, in step $i+1$, we form $H_{i+1}$ by replacing some element $h_j$ of a basis $B_i$ of $H_i$ by an element $h'_j \in H_i$. That is, $H_{i+1}$ is the image of the $H_i$-endomorphism which sends $h_j$ to $h'_j$ and all other basis elements are mapped to themselves. We stop after finitely-many steps and the resulting subgroup is $H$. It is easy to see that every echelon subgroup can be formed through a finite number of such transformations.

We also note, that the class of inert subgroups is still larger since e.g. every subgroup of rank 2 is inert, and there are subgroups of rank 2 that are not 1-gen-endo subgroups.

Let $\Aut(F)$ be the group of automorphisms of $F$. A subgroup $H\leqslant F$ is called \emph{1-auto-fixed} if $H=\Fix (\alpha)$, the subgroup of elements of $F$ fixed by $\alpha$, for some $\alpha\in \Aut(F)$; see~\cite{MV03}. Clearly, this class includes free factors of $F$.
In Dicks--Ventura~\cite{DV96} it was proved that 1-auto-fixed subgroups of finitely generated free groups are inert (recently, this has been generalized to 1-endo-fixed subgroups by Antol\'{\i}n--Jaikin~\cite{AJ21}, solving the inertia conjecture from~\cite{DV96}). Also, an explicit description of 1-auto-fixed subgroups of free groups was given in Martino--Ventura~\cite{MV04}, and from this it follows that 1-auto-fixed subgroups are in echelon form w.r.t. an appropriate ambient basis; see also~\cite{Ros13}.

\begin{figure}
\centering
\begin{tikzpicture}[shorten >=1pt, node distance=1.2 and 1.2, on grid,auto,>=stealth',decoration={snake, segment length=2mm, amplitude=0.5mm,post length=1.5mm}]
\begin{scope}[scale=0.55]
    \draw[] (0,1) circle [radius=1.5];
	\draw[] (0,2) circle [radius=2.5];
	\draw[] (0,3) circle [radius=3.5];
	\draw[] (0,4) circle [radius=4.5];
	\draw[] (0,5) circle [radius=5.5];
	\draw[] (0,6) circle [radius=6.5];
	\node at (0,1.5) {free};
	\node at (0,0.8) {factor};
	\node at (0,3.2) {1-auto-fixed};
	\node at (0,5.3) {echelon};
	\node at (0,7.3) {1-gen-endo};
	\node at (0,9.3) {inert};
	\node at (0,11.3) {compressed};
	\end{scope}
	\end{tikzpicture}
	\caption{A chain of classes of subgroups of a free group}
	\label{fig:subgpschain}
\end{figure}
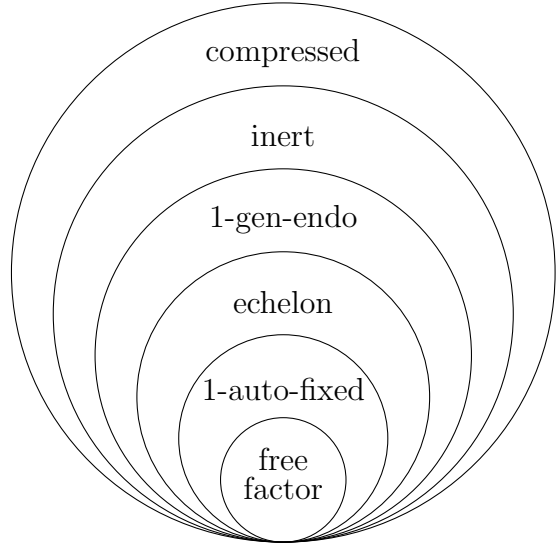

Next, we examine which of these classes are closed under the dependence operator. For some it is not difficult to give a positive answer: free factor, echelon and compressed. For the others, it is still an open problem.

\subsection{Free factors}

We start with the simplest class, that of free factors of $F$. Clearly, a subgroup which is a free factor is dependence-closed, so that it is invariant under the dependence operator. In the following proposition we are dealing with a somewhat more complex situation, where Stallings automata play an essential role in the proof.

\begin{proposition}\label{prop:free_product}
Let $F_1=F(A_1)$, $F_2=F(A_2)$, where $A_1$ and $A_2$ are disjoint alphabets, and let $F=F(A_1\sqcup A_2)=F_1*F_2$. Let $H_1\leqslant F_1$, $H_2\leqslant F_2$ and let $H=\gen{H_1, H_2}$. Then, $\Dep_{F}(H)=\Dep_{F_1}(H_1)* \Dep_{F_2}(H_2)$.
	
In particular, if $H_1$ is dependence-closed in $F_1$ and $H_2$ is dependence-closed in $F_2$, then $H$ is dependence-closed in $F$.
\end{proposition}

\begin{proof}
It is clear that $H=H_1*H_2$. The Stallings automaton $\Gamma_{A_1\sqcup A_2}(H)$ with respect to the basis $A_1\sqcup A_2$ of $F$ is the disjoint union of $\Gamma_{A_1\sqcup A_2}(H_1)=\Gamma_{A_1}(H_1)$ and $\Gamma_{A_1\sqcup A_2}(H_2)=\Gamma_{A_2}(H_2)$, with the basepoints $\bp_1$ and $\bp_2$ identified into a single one $\bp$. Let $R_1$, $R_2$, and $R$ be the sets of pairs of vertices as in~\eqref{eq:P}, respectively, from $\Gamma_{A_1}(H_1)$, $\Gamma_{A_2}(H_2)$, and $\Gamma_{A_1\sqcup A_2}(H)$. We claim that
 $$
R=R_1 \sqcup R_2.
 $$
Clearly, $R \supseteq R_1 \sqcup R_2$ since each of the pairs $(p,q)$ in $R_1$ and in $R_2$ results in a rank reduction after identifying $p$ with $q$ in the corresponding part. On the other hand, when we identify $\bp_1 \neq p\in \Gamma_{A_1}(H_1)$ and $\bp_2 \neq q\in \Gamma_{A_2}(H_2)$, we already obtain a deterministic $(A_1\sqcup A_2)$-automaton (by the fact that $A_1 \cap A_2=\emptyset$) and thus, no folding occurs and the rank increases by 1, implying that $(p,q) \notin R$. It follows that
 $$
\dep_{F}(H)=\bigcup_{(p,q) \in R_1 \sqcup R_2 \sqcup \{(\bp,\bp)\}} Hg_{(p,q)}H
 $$
and $\Dep_{F}(H)=\Dep_{F_1}(H_1)*\Dep_{F_2}(H_2)$.
\end{proof}

\subsection{The dependence subgroup of an 1-auto-fixed subgroup}

In the following example we show that the dependence subgroup of $H=\Fix(\alpha)$, for $\alpha \in \Aut(F)$, is not necessarily fixed by $\alpha$. This does not exclude the possibility of $\Dep(H)$ being fixed by another automorphism, as is indeed the case in this example.

\begin{example}
Let $F=F({a,b})$ and let $\alpha \colon F\to F$ be the automorphism defined by $a\mapsto a$, $b\mapsto ba$. Then, $H=\Fix(\alpha)=\gen{a, bab^{-1}}$, and $\Dep(H)=F$ is not fixed by $\alpha$. However, clearly $F$ is fixed by the identity automorphism.
\end{example}

\subsection{The dependence subgroup of an echelon subgroup}

For the class of echelon subgroups, we show that it is closed under the dependence operator.

\begin{proposition}
Let $F(A)$ be the free group on $A=\{a_1, \ldots ,a_r\}$, and let $H\leqfg F(A)$ be an echelon subgroup w.r.t. $A$. Then, $\Dep(H)$ is also an echelon subgroup w.r.t. $A$.
\end{proposition}

\begin{proof}
Let $F_0 =\{1\}$ and $F_i=\gen{a_1,\ldots ,a_i}\leqff F$, for $i=1,\ldots ,r$. Also, let $H_i=H\cap F_i\leqff H$ and $D_i=\Dep(H)\cap F_i\leqff \Dep(H)$, for $i=0,\ldots ,r$. Since $F_{i-1} \leqff F_i$ then $H_{i-1}\leqff H_i$, $D_{i-1}\leqff D_i$ and $\rk(H_{i-1}) \leqslant \rk(H_i)$, $\rk(D_{i-1}) \leqslant \rk(D_i)$, for $i=1,\ldots ,r$. Moreover, observe that when $\rk(D_i)=\rk(D_{i-1})$ then $D_i=D_{i-1}$ and this implies that $H_i=H_{i-1}$. Now, suppose that $H$ is in echelon form w.r.t. $A$, i.e., $0\leqslant \rk(H_i)-\rk(H_{i-1})\leqslant 1$, for each $i$. Then $\rk(H)$ equals the number of times that $\rk(H_i)=\rk(H_{i-1})+1$. Since an equality in ranks, $\rk(D_i)=\rk(D_{i-1})$, implies $H_i=H_{i-1}$, we deduce that $\rk(D_i)-\rk(D_{i-1}) \geqslant \rk(H_i)-\rk(H_{i-1})$, for each $i=1,\ldots ,r$. Therefore,
 $$
\rk(\Dep(H))=\sum_{i=1}^r (\rk(D_i)-\rk(D_{i-1}))\geqslant \sum_{i=1}^r (\rk(H_i)-\rk(H_{i-1})) =\rk(H).
 $$
But, by Proposition~\ref{prop:rankDep}, $\rk(\Dep(H))\leqslant \rk(H)$. Hence, the above inequality is an equality and thus, for each $i=1, \ldots ,r$,
 $$
0\leqslant \rk(D_i)-\rk(D_{i-1})=\rk(H_i)-\rk(H_{i-1})\leqslant 1.
 $$
Thus, $\Dep(H)$ is in echelon form w.r.t $A$.
\end{proof}

\subsection{The dependence subgroup of a compressed subgroup}

\begin{proposition}
Let $F$ be a free group and let $H\leqfg F$ be a compressed subgroup. Then, $\Dep(H)$ is compressed, with $\rk(\Dep(H)) = \rk(H)$.
\end{proposition}

\begin{proof}
By Proposition~\ref{prop:rankDep}, $\rk(\Dep(H))\leqslant \rk(H)$, but since $H\leqslant \Dep(H)$ and $H$ is compressed then $\rk(\Dep(H))\geqslant \rk(H)$. It follows that $\rk(\Dep(H))=\rk(H)$. Let now $K$ be an arbitrary subgroup with $\Dep(H) \leqslant K \leqslant F$. Then, since $H \leqslant K$ and $H$ is compressed, we have
 $$
\rk(K) \geqslant \rk(H)= \rk(\Dep(H)),
 $$
which proves that $\Dep(H)$ is compressed.
\end{proof}

\section{Dependence sequence and dependence closure}
\label{sec:dep_closure}

In this section we investigate further the dependence operator and the notion of dependence closure in the setting of a free group $F$. Recall that a subgroup $H\leqslant F$ is dependence-closed in $F$ if $\Dep_{F}(H)=H$. For example, free factors $H\leqff F$ are dependence-closed in $F$, but they are not the only ones (e.g., $H=\langle a^2b^2\rangle$ and $K=\langle [a,b]\rangle$ are both dependence-closed in $F(\{a,b\})$).

We provide here an example of a subgroup $H\leqfg F(a,b)$, such that $\Dep(H)$ is \emph{not} dependence-closed, that is, such that the process of adding to $H$ \emph{all} the elements that depend on $H$ and constructing $\Dep(H)$, itself creates new dependent elements. In fact, for every $m\geqslant 1$, we exhibit a finitely generated subgroup $H_m\leqfg F(a,b)$ with dependence length $\depl(H_m)=m$, i.e., for which its dependence sequence has length exactly $m$, $H_m<\Dep^1(H_m)<\Dep^2(H_m)<\cdots <\Dep^m(H_m)=\DepCl(H_m)$.

\begin{example}\label{ex:dep_length}
For every $m\geqslant 1$, we construct a subgroup $H_m \leqslant F=F(a,b)$ of dependence length $m\geqslant 1$ in the following way (see Figure~\ref{fig:depseq} for the case $m=4$).

Let $p_1,p_2,p_3,\ldots = 2,3,5,\ldots$ be the ascending sequence of prime numbers. We define the elements $A_n \in F$, $n \ge 1$, to be\footnote{The fact that the powers of $a$ and $b$ in the $A_n$ are prime numbers may reduce the number of possible foldings in the graphs of the dependent subgroups, but it does not seem to be a necessary condition.}
 $$
A_1=a^2b^3a^5b^7, \; A_2=a^{11}b^{13}a^{17}b^{19},\ldots, \; A_n=a^{\,p_{4n-3}}\,b^{\,p_{4n-2}} \,a^{\,p_{4n-1}} \,b^{\,p_{4n}}.
 $$
 
Now, let $H_m=\gen{h_1,\ldots,h_{m-1},A_{m}^5,A_{m+1}} \leqslant F$, where
$$
	h_1=A_2 A_3 A_1^5 A_3 A_2, \; h_2=A_3 A_4 A_2^5 A_4 A_3, \ldots, \; h_{m-1}=A_m A_{m+1} A_{m-1}^5 A_{m+1} A_m.
$$

\begin{proposition}
	$H_m$ is of dependence length $m$.
	\label{pr:depseq}
\end{proposition}
\begin{proof}
We claim that
\begin{align*}
	&\Dep(H_m)=\gen{h_1,\ldots,h_{m-2},A_{m-1}^5, A_{m}, A_{m+1}},\\
	&\Dep^2(H_m)=\gen{h_1,\ldots,h_{m-3}, A_{m-2}^5, A_{m-1},A_{m}, A_{m+1}},\\ 
	&\quad \vdots \\
	&\Dep^j(H_m) =\gen{h_1,\ldots,h_{m-1-j},A_{m-j}^5,A_{m-j+1},\ldots,A_{m+1}}, \\
	&\quad \vdots \\
	&\Dep^m(H_m) =\gen{A_1, A_2, \ldots,A_{m+1}},
\end{align*}
where all the above generators are free generators.

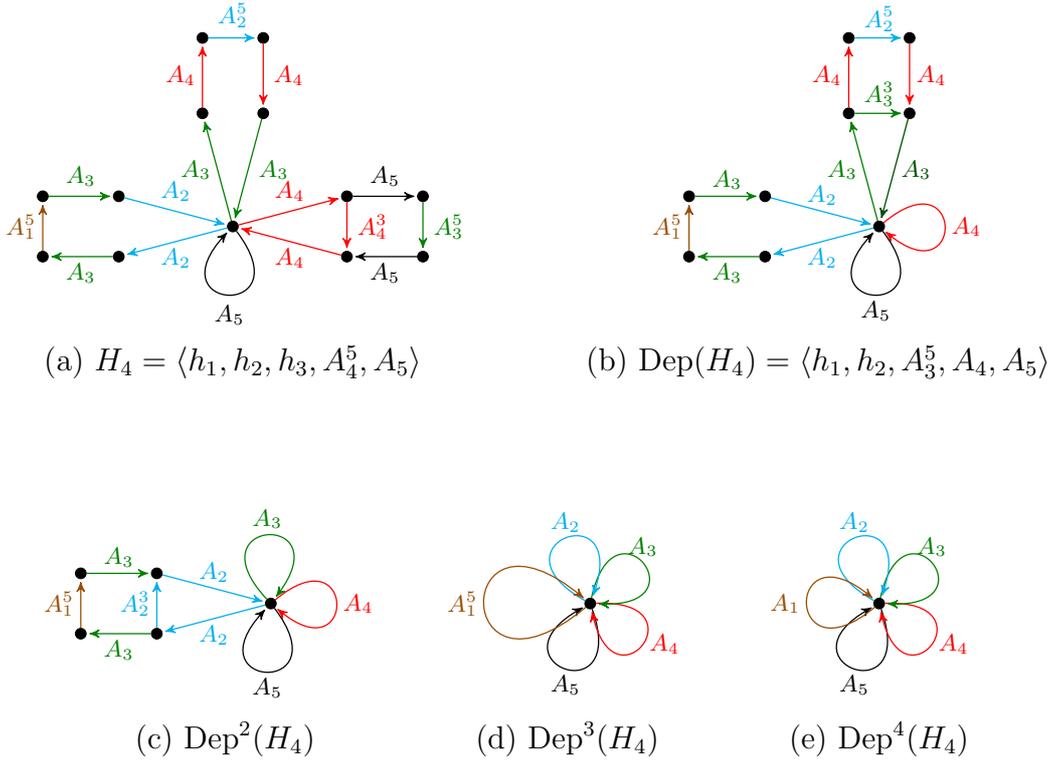
\begin{figure}
\centering
\begin{tikzpicture}[shorten >=1pt, node distance=1cm and 1.5cm, on grid,>=stealth',scale=1]
\tikzset{every loop/.style={min distance=10mm,in=0,out=60,looseness=5}}
\begin{scope}
    \node[state] (Hroot) {};
	\path[->]
	(Hroot) edge[loop,rotate=90,out=-140,in=140,looseness=40] node[pos=0.5,below=0.05] {$A_5$} (Hroot);

	\node[state] (Ha1) [below left = 0.4 and 1.5 of Hroot] {};
	\node[state] (Ha2) [left = 1. of Ha1] {};
	\node[state] (Ha3) [above = 0.8 of Ha2] {};
	\node[state] (Ha4) [above left = 0.4 and 1.5 of Hroot] {};
	\path[->]
	(Hroot) edge[cyan] node[pos=0.5,below] {$A_2$} (Ha1)
	(Ha1) edge[green!50!black] node[pos=0.5,below] {$A_3$} (Ha2)
	(Ha2) edge[orange!60!black] node[pos=0.5,left] {$A_1^5$} (Ha3)
	(Ha3) edge[green!50!black] node[pos=0.5,above=0.05] {$A_3$} (Ha4)
	(Ha4) edge[cyan] node[pos=0.5,above=0.05] {$A_2$} (Hroot);

	\node[state] (Hb1) [above left = 1.5 and 0.4 of Hroot] {};
	\node[state] (Hb2) [above = 1. of Hb1] {};
	\node[state] (Hb3) [right = 0.8 of Hb2] {};
	\node[state] (Hb4) [above right = 1.5 and 0.4 of Hroot] {};
	\path[->]
	(Hroot) edge[green!50!black] node[pos=0.5,left] {$A_3$} (Hb1)
	(Hb1) edge[red] node[pos=0.5,left] {$A_4$} (Hb2)
	(Hb2) edge[cyan] node[pos=0.5,above=0.05] {$A_2^5$} (Hb3)
	(Hb3) edge[red] node[pos=0.5,right=0.05] {$A_4$} (Hb4)
	(Hb4) edge[green!50!black] node[pos=0.5,right=0.05] {$A_3$} (Hroot);

	\node[state] (Hc1) [above right = 0.4 and 1.5 of Hroot] {};
	\node[state] (Hc2) [right = 1. of Hc1] {};
	\node[state] (Hc3) [below = 0.8 of Hc2] {};
	\node[state] (Hc4) [below right = 0.4 and 1.5 of Hroot] {};
	\path[->]
	(Hroot) edge[red] node[pos=0.5,above=0.05] {$A_4$} (Hc1)
	(Hc1) edge[] node[pos=0.5,above=0.05] {$A_5$} (Hc2)
	(Hc2) edge[green!50!black] node[pos=0.5,right=0.05] {$A_3^5$} (Hc3)
	(Hc3) edge[] node[pos=0.5,below] {$A_5$} (Hc4)
	(Hc4) edge[red] node[pos=0.5,below] {$A_4$} (Hroot)
	(Hc1) edge[red] node[pos=0.5,right=0.05] {$A_4^3$} (Hc4);
	\draw (0,-1.8) node {(a) $H_4=\gen{h_1,h_2,h_3,A_4^5,A_5}$};
\end{scope}

\begin{scope}[xshift=8.5cm]
	\node[state] (DHroot) {};
	\path[->]
	(DHroot) edge[loop,rotate=90,out=-140,in=140,looseness=40] node[pos=0.5,below] {$A_5$} (DHroot)
	(DHroot) edge[loop,rotate=180,out=-140,in=140,looseness=40,red] node[pos=0.5,right] {$A_4$} (DHroot);
	
	\node[state] (DHa1) [below left = 0.4 and 1.5 of DHroot] {};
	\node[state] (DHa2) [left = 1. of DHa1] {};
	\node[state] (DHa3) [above = 0.8 of DHa2] {};
	\node[state] (DHa4) [above left = 0.4 and 1.5 of DHroot] {};
	\path[->]
	(DHroot) edge[cyan] node[pos=0.5,below] {$A_2$} (DHa1)
	(DHa1) edge[green!50!black] node[pos=0.5,below] {$A_3$} (DHa2)
	(DHa2) edge[orange!60!black] node[pos=0.5,left] {$A_1^5$} (DHa3)
	(DHa3) edge[green!50!black] node[pos=0.5,above] {$A_3$} (DHa4)
	(DHa4) edge[cyan] node[pos=0.5,above] {$A_2$} (DHroot);
	
	\node[state] (DHb1) [above left = 1.5 and 0.4 of DHroot] {};
	\node[state] (DHb2) [above = 1. of DHb1] {};
	\node[state] (DHb3) [right = 0.8 of DHb2] {};
	\node[state] (DHb4) [above right = 1.5 and 0.4 of DHroot] {};
	\path[->]
	(DHroot) edge[green!50!black] node[pos=0.5,left] {$A_3$} (DHb1)
	(DHb1) edge[red] node[pos=0.5,left] {$A_4$} (DHb2)
	(DHb2) edge[cyan] node[pos=0.5,above] {$A_2^5$} (DHb3)
	(DHb3) edge[red] node[pos=0.5,right] {$A_4$} (DHb4)
	(DHb4) edge[green!50!black!60!black] node[pos=0.5,right] {$A_3$} (DHroot)
	(DHb1) edge[green!50!black] node[pos=0.5,above] {$A_3^3$} (DHb4);
	
	\draw (-0.8,-1.8) node {(b) $\Dep(H_4)=\gen{h_1,h_2,A_3^5,A_4,A_5}$};
\end{scope}

\begin{scope}[xshift=0.5cm, yshift=-5cm]
	\node[state] (D2Hroot) {};
	\path[->]
	(D2Hroot) edge[loop,rotate=90,out=-140,in=140,looseness=40] node[pos=0.5,below] {$A_5$} (D2Hroot)
	(D2Hroot) edge[loop,rotate=180,out=-140,in=140,looseness=40,red] node[pos=0.5,right] {$A_4$} (D2Hroot)
	(D2Hroot) edge[loop,rotate=270,out=-140,in=140,looseness=40,green!50!black] node[pos=0.5,above] {$A_3$} (D2Hroot);
	
	\node[state] (D2Ha1) [below left = 0.4 and 1.5 of D2Hroot] {};
	\node[state] (D2Ha2) [left = 1. of D2Ha1] {};
	\node[state] (D2Ha3) [above = 0.8 of D2Ha2] {};
	\node[state] (D2Ha4) [above left = 0.4 and 1.5 of D2Hroot] {};
	\path[->]
	(D2Hroot) edge[cyan] node[pos=0.5,below] {$A_2$} (D2Ha1)
	(D2Ha1) edge[green!50!black] node[pos=0.5,below] {$A_3$} (D2Ha2)
	(D2Ha2) edge[orange!60!black] node[pos=0.5,left] {$A_1^5$} (D2Ha3)
	(D2Ha3) edge[green!50!black] node[pos=0.5,above] {$A_3$} (D2Ha4)
	(D2Ha4) edge[cyan] node[pos=0.5,above] {$A_2$} (D2Hroot)
	(D2Ha1) edge[cyan] node[pos=0.5,left] {$A_2^3$} (D2Ha4);
	
	\draw (-0.6,-1.8) node {(c) $\Dep^2(H_4)$};
\end{scope}

\begin{scope}[xshift=4.7cm,yshift=-5cm]
	\node[state] (D3Hroot) {};
	\path[->]
	(D3Hroot) edge[loop,rotate=72,out=-140,in=140,looseness=40] node[pos=0.5,below] {$A_5$} (D3Hroot)
	(D3Hroot) edge[loop,rotate=144,out=-140,in=140,looseness=40,red] node[pos=0.5,right] {$A_4$} (D3Hroot)
	(D3Hroot) edge[loop,rotate=216,out=-140,in=140,looseness=40,green!50!black] node[pos=0.5,above] {$A_3$} (D3Hroot)
	(D3Hroot) edge[loop,rotate=288,out=-140,in=140,looseness=40,cyan] node[pos=0.5,above] {$A_2$} (D3Hroot)
	(D3Hroot) edge[loop,out=-140,in=140,looseness=60,orange!60!black] node[pos=0.5,left] {$A_1^5$} (D3Hroot);
	
	\draw (-0.3,-1.8) node {(d) $\Dep^3(H_4)$};
\end{scope}

\begin{scope}[xshift=8.5cm,yshift=-5cm]
	\node[state] (D4Hroot) {};
	\path[->]
	(D4Hroot) edge[loop,rotate=72,out=-140,in=140,looseness=40] node[pos=0.5,below] {$A_5$} (D4Hroot)
	(D4Hroot) edge[loop,rotate=144,out=-140,in=140,looseness=40,red] node[pos=0.5,right] {$A_4$} (D4Hroot)
	(D4Hroot) edge[loop,rotate=216,out=-140,in=140,looseness=40,green!50!black] node[pos=0.5,above] {$A_3$} (D4Hroot)
	(D4Hroot) edge[loop,rotate=288,out=-140,in=140,looseness=40,cyan] node[pos=0.5,above] {$A_2$} (D4Hroot)
	(D4Hroot) edge[loop,out=-140,in=140,looseness=40,orange!60!black] node[pos=0.5,left] {$A_1$} (D4Hroot);
	
	\draw (0,-1.8) node {(e) $\Dep^4(H_4)$};
\end{scope}

\end{tikzpicture}
	\caption{The dependence sequence of length $4$ of the subgroup $H_4$, where $h_{j}=A_{j+1} A_{j+2} A_{j}^5 A_{j+2} A_{j+1}$, for $j=1,2,3$.}
	\label{fig:depseq}
\end{figure}

The idea in constructing $H_m$ is the following. Since $H_m$ contains the element $A_m^5$, then, clearly, $A_m \in \Dep(H_m)$. Consequently, $A_{m-1}^5$ is ``revealed" since $\Dep(H_m)$ contains $A_{m+1}$, $A_m$ and $h_{m-1}=A_m A_{m+1} A_{m-1}^5 A_{m+1} A_m$. Thus, $A_{m-1}$ is a dependent element of $\Dep(H_m)$. In the next step, we obtain $A_{m-2}$ as a dependent element of $\Dep^2(H_m)$ through $A_m$, $A_{m-1}$ and $h_{m-2}=A_{m-1} A_m A_{m-2}^5 A_m A_{m-1}$, and so on. It follows that after (at most) $m$ steps we end up with the subgroup $\Dep^m(H_m) =\gen{A_1, A_2, \ldots,A_{m+1}}$, which is dependence-closed.

	
It is, however, less clear that throughout the dependence sequence there are no more dependent elements besides the ones mentioned above, which can cause additional foldings in the graphs of the subgroups. So, let us examine the graph of a subgroup $\Dep^{m-j}(H_m)$. It contains $m+1$ cycles that are of (at most) four types: the $\bp$-cycles
$h_{i}=A_{i+1} A_{i+2} A_{i}^5 A_{i+2} A_{i+1}$, for $i=1,\ldots, j-1$, $A_k$, for $k=j+1,\ldots,m+1$, and $A_{j}^5$, in addition to the cycle $A_{j+1} A_{j-1}^5 A_{j+1} A_{j}^{-3}$ that does not include the basepoint and is formed through $h_{j-1}$ and $A_{j}^5$.
For convenience and w.l.o.g., we will make use  of the graph of $H_4$ (Figure~\ref{fig:identify}(a)) when checking the possible identifications of vertices. This graph contains cycles of all the above four types and represent.

The fact that there are not many dependent elements is also due to the fact that the vertices are of low degrees. Most of the vertices are of degree $2$, with outgoing pairs of vertices $(a,a^{-1})$, $(b,b^{-1})$, $(a,b^{-1})$ or $(b,a^{-1})$. Some vertices are of degree $3$ (at most $2(m+1)$ such vertices in $\dep^i(H_m)$), with outgoing triples of vertices $(a,a^{-1},b)$ or $(b,b^{-1},a^{-1})$. The latter are to be found in the neighbourhood of the basepoint $\bp$, and two more vertices of degree $3$ are located near the cycle that does not pass through $\bp$ (near vertices 9 and 12 in Figure~\ref{fig:identify}(a)).
%

Let us now look at the possible identifications of vertices.
We call a vertex at which an $A_i$-segment starts or terminates a \emph{prime} vertex, and otherwise - a \emph{secondary} vertx.
When no prime vertex is involved in the identification (either directly or after the consequent foldings) then it takes place within some $A_i$. But clearly any such identification forms a loop that cannot be erased.

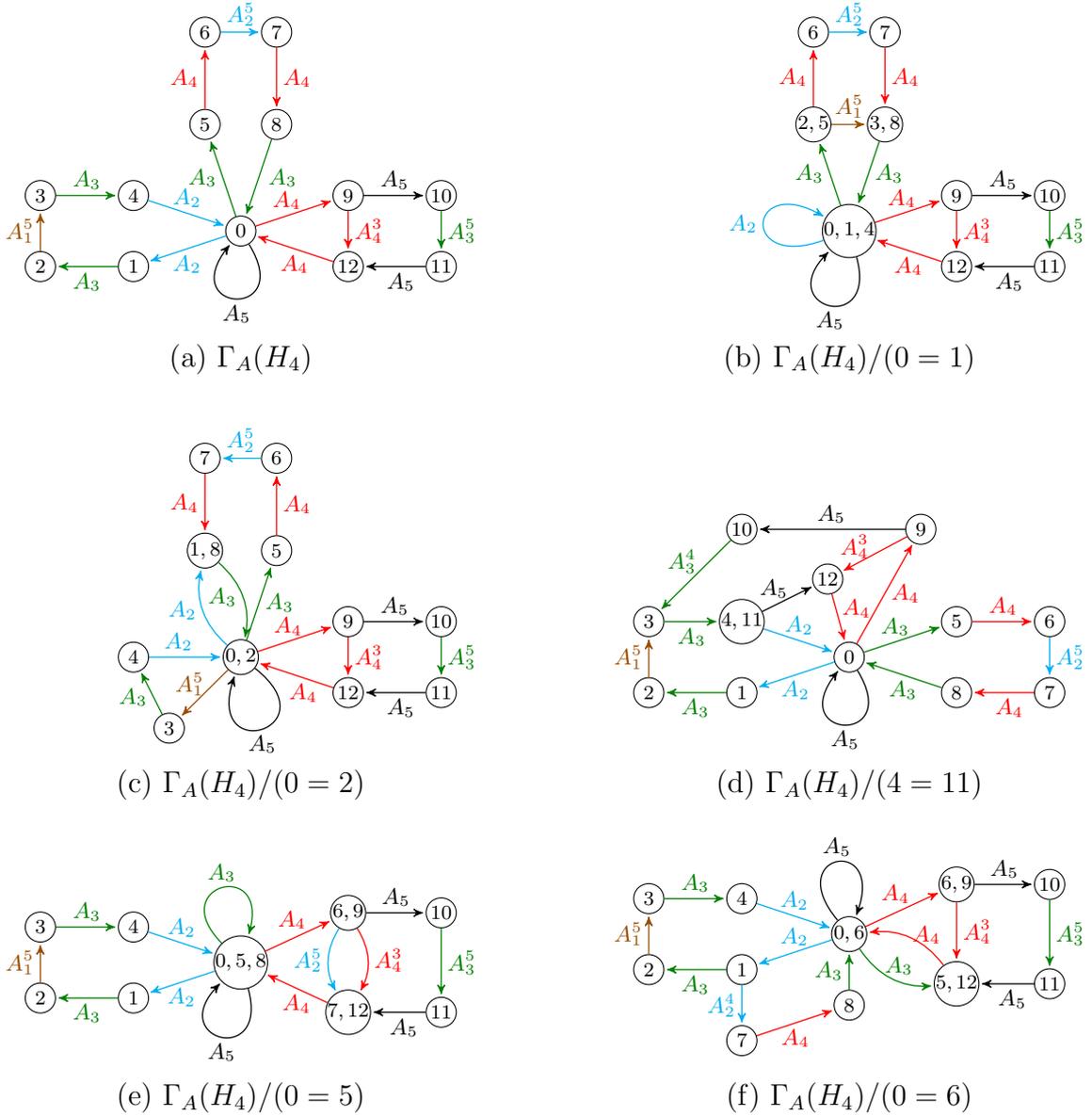
\begin{figure}
	\centering
	\begin{tikzpicture}[shorten >=1pt, node distance=1cm and 1.5cm, on grid,>=stealth',scale=1]
		\tikzset{every loop/.style={min distance=10mm, in=0, out=60, looseness=5}}
		\begin{scope}
			\vertex (v0) {$0$};
			\path[->]
			(v0) edge[loop,rotate=90,out=-140,in=150,looseness=15] node[pos=0.5,below] {$A_5$} (v0);
			\vertex (v1) [below left = 0.5 and 1.5 of v0] {$1$};
			\vertex (v2) [left = 1.3 of v1] {$2$};
			\vertex (v3) [above = 1.0 of v2] {$3$};
			\vertex (v4) [above left = 0.5 and 1.5 of v0] {$4$};
			
			\path[->]
			(v0) edge[cyan] node[pos=0.5,below] {$A_2$} (v1)
			(v1) edge[green!50!black] node[pos=0.5,below] {$A_3$} (v2)
			(v2) edge[orange!60!black] node[pos=0.5,left] {$A_1^5$} (v3)
			(v3) edge[green!50!black] node[pos=0.5,above] {$A_3$} (v4)
			(v4) edge[cyan] node[pos=0.5,above] {$A_2$} (v0);
			
			\vertex (v5) [above left = 1.5 and 0.5 of v0] {$5$};
			\vertex (v6) [above = 1.3 of v5] {$6$};
			\vertex (v7) [right = 1.0 of v6] {$7$};
			\vertex (v8) [above right = 1.5 and 0.5 of v0] {$8$};
			
			\path[->]
			(v0) edge[green!50!black] node[pos=0.5,left] {$A_3$} (v5)
			(v5) edge[red] node[pos=0.5,left] {$A_4$} (v6)
			(v6) edge[cyan] node[pos=0.5,above] {$A_2^5$} (v7)
			(v7) edge[red] node[pos=0.5,right] {$A_4$} (v8)
			(v8) edge[green!50!black] node[pos=0.5,right=0.05] {$A_3$} (v0);
			
			\vertex (v9) [above right = 0.5 and 1.5 of v0] {$9$};
			\vertex (v10) [right = 1.3 of v9] {$10$};
			\vertex (v11) [below = 1.0 of v10] {$11$};
			\vertex (v12) [below right = 0.5 and 1.5 of v0] {$12$};
			
			\path[->]
			(v0) edge[red] node[pos=0.4,above] {$A_4$} (v9)
			(v9) edge[] node[pos=0.5,above] {$A_5$} (v10)
			(v10) edge[green!50!black] node[pos=0.5,right] {$A_3^5$} (v11)
			(v11) edge[] node[pos=0.4,below] {$A_5$} (v12)
			(v12) edge[red] node[pos=0.5,below] {$A_4$} (v0)
			(v9) edge[red] node[pos=0.5,right] {$A_4^3$} (v12);
			\draw (0,-1.8) node {(a) $\Gamma_A(H_4)$};
		\end{scope}
		
		\begin{scope}[xshift=8.5cm]
			\vertex (v0) {$0,1,4$};
			\path[->]
			(v0) edge[loop,rotate=90,out=-150,in=150,looseness=10] node[pos=0.5,below] {$A_5$} (v0)
			(v0) edge[cyan,loop,rotate=0,out=-160,in=150,looseness=10] node[pos=0.5,left] {$A_2$} (v0);
			
			\vertex (v5) [above left = 1.5 and 0.5 of v0] {$2,5$};
			\vertex (v6) [above = 1.3 of v5] {$6$};
			\vertex (v7) [right = 1.0 of v6] {$7$};
			\vertex (v8) [above right = 1.5 and 0.5 of v0] {$3,8$};
			
			\path[->]
			(v0) edge[green!50!black] node[pos=0.5,left] {$A_3$} (v5)
			(v5) edge[red] node[pos=0.5,left] {$A_4$} (v6)
			(v6) edge[cyan] node[pos=0.5,above] {$A_2^5$} (v7)
			(v5) edge[orange!60!black] node[pos=0.5,above] {$A_1^5$} (v8)
			(v7) edge[red] node[pos=0.5,right] {$A_4$} (v8)
			(v8) edge[green!50!black] node[pos=0.5,right] {$A_3$} (v0);
			
			\vertex (v9) [above right = 0.5 and 1.5 of v0] {$9$};
			\vertex (v10) [right = 1.3 of v9] {$10$};
			\vertex (v11) [below = 1.0 of v10] {$11$};
			\vertex (v12) [below right = 0.5 and 1.5 of v0] {$12$};
			
			\path[->]
			(v0) edge[red] node[pos=0.4,above] {$A_4$} (v9)
			(v9) edge[] node[pos=0.5,above] {$A_5$} (v10)
			(v10) edge[green!50!black] node[pos=0.5,right] {$A_3^5$} (v11)
			(v11) edge[] node[pos=0.4,below] {$A_5$} (v12)
			(v12) edge[red] node[pos=0.5,below] {$A_4$} (v0)
			(v9) edge[red] node[pos=0.5,right] {$A_4^3$} (v12);
			\draw (0,-1.8) node {(b) $\Gamma_A(H_4)/(0=1)$};
		\end{scope}
		
		\begin{scope}[yshift=-6cm]]
			\vertex (v0) {$0,2$};
			\path[->]
			(v0) edge[loop,rotate=110,out=-140,in=150,looseness=15] node[pos=0.5,below] {$A_5$} (v0);
			\vertex (v3) [below left = 1.0 and 1.0 of v0] {$3$};
			\vertex (v4) [above left = 0.0 and 1.5 of v0] {$4$};
			
			\path[->]
			(v0) edge[orange!60!black] node[pos=0.3,left=0.05] {$A_1^5$} (v3)
			(v3) edge[green!50!black] node[pos=0.4,left] {$A_3$} (v4)
			(v4) edge[cyan] node[pos=0.4,above] {$A_2$} (v0);
			
			\vertex (v8) [above left = 1.5 and 0.5 of v0] {$1,8$};
			\vertex (v7) [above = 1.3 of v8] {$7$};
			\vertex (v6) [right = 1.0 of v7] {$6$};
			\vertex (v5) [above right = 1.5 and 0.5 of v0] {$5$};
			
			\path[->]
			(v0) edge[bend left, cyan] node[pos=0.5,left] {$A_2$} (v8)
			(v8) edge[bend left, green!50!black] node[pos=0.5,left] {$A_3$} (v0)
			(v7) edge[red] node[pos=0.5,left] {$A_4$} (v8)
			(v6) edge[cyan] node[pos=0.5,above] {$A_2^5$} (v7)
			(v5) edge[red] node[pos=0.5,right] {$A_4$} (v6)
			(v0) edge[green!50!black] node[pos=0.5,right] {$A_3$} (v5);
			
			\vertex (v9) [above right = 0.5 and 1.5 of v0] {$9$};
			\vertex (v10) [right = 1.3 of v9] {$10$};
			\vertex (v11) [below = 1.0 of v10] {$11$};
			\vertex (v12) [below right = 0.5 and 1.5 of v0] {$12$};
			
			\path[->]
			(v0) edge[red] node[pos=0.4,above] {$A_4$} (v9)
			(v9) edge[] node[pos=0.5,above] {$A_5$} (v10)
			(v10) edge[green!50!black] node[pos=0.5,right] {$A_3^5$} (v11)
			(v11) edge[] node[pos=0.4,below] {$A_5$} (v12)
			(v12) edge[red] node[pos=0.4,below] {$A_4$} (v0)
			(v9) edge[red] node[pos=0.5,right] {$A_4^3$} (v12);
			\draw (0,-1.8) node {(c) $\Gamma_A(H_4)/(0=2)$};
		\end{scope}

		\begin{scope}[xshift=8.5cm, yshift=-6cm]
			\vertex (v0) {$0$};
			\path[->]
			(v0) edge[loop,rotate=90,out=-140,in=150,looseness=15] node[pos=0.5,below] {$A_5$} (v0);
			\vertex (v1) [below left = 0.5 and 1.5 of v0] {$1$};
			\vertex (v2) [left = 1.3 of v1] {$2$};
			\vertex (v3) [above = 1.0 of v2] {$3$};
			\vertex (v4) [above left = 0.5 and 1.5 of v0] {$4,11$};
			
			\path[->]
			(v0) edge[cyan] node[pos=0.5,below] {$A_2$} (v1)
			(v1) edge[green!50!black] node[pos=0.5,below] {$A_3$} (v2)
			(v2) edge[orange!60!black] node[pos=0.5,left] {$A_1^5$} (v3)
			(v3) edge[green!50!black] node[pos=0.5,below] {$A_3$} (v4)
			(v4) edge[cyan] node[pos=0.5,above] {$A_2$} (v0);
			
			\vertex (v5) [above right = 0.5 and 1.5 of v0] {$5$};
			\vertex (v6) [right = 1.3 of v5] {$6$};
			\vertex (v7) [below = 1.0 of v6] {$7$};
			\vertex (v8) [below right = 0.5 and 1.5 of v0] {$8$};
			
			\path[->]
			(v0) edge[green!50!black] node[pos=0.4,above] {$A_3$} (v5)
			(v5) edge[red] node[pos=0.5,above] {$A_4$} (v6)
			(v6) edge[cyan] node[pos=0.5,right] {$A_2^5$} (v7)
			(v7) edge[red] node[pos=0.4,below] {$A_4$} (v8)
			(v8) edge[green!50!black] node[pos=0.5,below] {$A_3$} (v0);
			
			\vertex (v9) [above right = 1.8 and 1.0 of v0] {$9$};
			\vertex (v10) [above = 1.3 of v4] {$10$};
			\vertex (v12) [above left = 1.1 and 0.3 of v0] {$12$};
			
			\path[->]
			(v0) edge[red] node[pos=0.5,right] {$A_4$} (v9)
			(v9) edge[] node[pos=0.5,above] {$A_5$} (v10)
			(v10) edge[green!50!black] node[pos=0.3,left=0.1] {$A_3^4$} (v3)
			(v4) edge[] node[pos=0.9,left=0.2] {$A_5$} (v12)
			(v12) edge[red] node[pos=0.3,right] {$A_4$} (v0)
			(v9) edge[red] node[pos=0.8,above] {$A_4^3$} (v12);
			\draw (0,-1.8) node {(d) $\Gamma_A(H_4)/(4=11)$};
		\end{scope}
		
		\begin{scope}[yshift=-10.3cm]
			\vertex (v0) {$0,5,8$};
			\path[->]
			(v0) edge[loop,rotate=90,out=-150,in=150,looseness=10] node[pos=0.5,below] {$A_5$} (v0);
			\vertex (v1) [below left = 0.5 and 1.5 of v0] {$1$};
			\vertex (v2) [left = 1.3 of v1] {$2$};
			\vertex (v3) [above = 1.0 of v2] {$3$};
			\vertex (v4) [above left = 0.5 and 1.5 of v0] {$4$};
			
			\path[->]
			(v0) edge[cyan] node[pos=0.5,below] {$A_2$} (v1)
			(v1) edge[green!50!black] node[pos=0.5,below] {$A_3$} (v2)
			(v2) edge[orange!60!black] node[pos=0.5,left] {$A_1^5$} (v3)
			(v3) edge[green!50!black] node[pos=0.5,above] {$A_3$} (v4)
			(v4) edge[cyan] node[pos=0.5,above] {$A_2$} (v0);
			
			\path[->]
			(v0) edge[loop, rotate=270, out=-150, in=150, looseness=10, green!50!black] node[pos=0.5,above] {$A_3$} (v0);
			
			\vertex (v9) [above right = 0.7 and 1.5 of v0] {$6,9$};
			\vertex (v10) [right = 1.3 of v9] {$10$};
			\vertex (v11) [below = 1.4 of v10] {$11$};
			\vertex (v12) [below right = 0.7 and 1.5 of v0] {$7,12$};
			
			\path[->]
			(v0) edge[red] node[pos=0.4,above=0.05] {$A_4$} (v9)
			(v9) edge[] node[pos=0.5,above] {$A_5$} (v10)
			(v10) edge[green!50!black] node[pos=0.5,right] {$A_3^5$} (v11)
			(v11) edge[] node[pos=0.4,below] {$A_5$} (v12)
			(v12) edge[red] node[pos=0.5,below] {$A_4$} (v0)
			(v9) edge[red,bend left] node[pos=0.5,right] {$A_4^3$} (v12)
			(v9) edge[cyan,bend right] node[pos=0.5,left] {$A_2^5$} (v12);
			\draw (0,-1.9) node {(e) $\Gamma_A(H_4)/(0=5)$};
		\end{scope}
		
		\begin{scope}[xshift=8.5cm, yshift=-9.9cm]
			\vertex (v0) {$0,6$};
			\path[->]
			(v0) edge[loop,rotate=270,out=-150,in=150,looseness=15] node[pos=0.5,above] {$A_5$} (v0);
			\vertex (v1) [below left = 0.5 and 1.5 of v0] {$1$};
			\vertex (v2) [left = 1.3 of v1] {$2$};
			\vertex (v3) [above = 1.0 of v2] {$3$};
			\vertex (v4) [above left = 0.5 and 1.5 of v0] {$4$};
			
			\path[->]
			(v0) edge[cyan] node[pos=0.5,above] {$A_2$} (v1)
			(v1) edge[green!50!black] node[pos=0.5,below] {$A_3$} (v2)
			(v2) edge[orange!60!black] node[pos=0.5,left] {$A_1^5$} (v3)
			(v3) edge[green!50!black] node[pos=0.5,above] {$A_3$} (v4)
			(v4) edge[cyan] node[pos=0.5,above] {$A_2$} (v0);
			
			\vertex (v7) [below right = 1.0 and 0.0 of v1] {$7$};
			\vertex (v8) [below left = 1.0 and 0.0 of v0] {$8$};
			
			\path[->]
			(v1) edge[cyan] node[pos=0.5,left] {$A_2^4$} (v7)
			(v7) edge[red] node[pos=0.5,below] {$A_4$} (v8)
			(v8) edge[green!50!black] node[pos=0.5,left] {$A_3$} (v0);
			
			\vertex (v9) [above right = 0.7 and 1.5 of v0] {$6,9$};
			\vertex (v10) [right = 1.3 of v9] {$10$};
			\vertex (v11) [below = 1.4 of v10] {$11$};
			\vertex (v12) [below right = 0.7 and 1.5 of v0] {$5,12$};
			
			\path[->]
			(v0) edge[red] node[pos=0.4,above=0.05] {$A_4$} (v9)
			(v9) edge[] node[pos=0.5,above] {$A_5$} (v10)
			(v10) edge[green!50!black] node[pos=0.5,right] {$A_3^5$} (v11)
			(v11) edge[] node[pos=0.4,below] {$A_5$} (v12)
			(v12) edge[red, bend right] node[pos=0.6,right=0.05] {$A_4$} (v0)
			(v0) edge[green!50!black, bend right] node[pos=0.6,above] {$A_3$} (v12)
			(v9) edge[red] node[pos=0.5,right] {$A_4^3$} (v12);
			\draw (0,-2.3) node {(f) $\Gamma_A(H_4)/(0=6)$};
		\end{scope}
	\end{tikzpicture}
	\caption{Identifying vertices of $\Gamma_A(H_4)$}
	\label{fig:identify}
\end{figure}

So, we can assume that the identification is of a prime vertex $p$ and another vertex $q$. Suppose, first, that $q$ is a secondary vertex. 
A $p$-cycle then starts with an $A_i$-segment (or an $A_i^{-1}$-segment), for some $i$, and ends with an $A_j$-segment (or an $A_j^{-1}$-segment), for some $j$.
Let us assume, w.l.o.g. that it starts with an $A_i$-segment and ends with an $A_j$-segment.
Given that $A_i=a^{\,p_{4i-3}}\,b^{\,p_{4i-2}} \,a^{\,p_{4i-1}} \,b^{\,p_{4i}}$, the $p$-cycle begins with a sequence of $a$'s and ends with a sequence of $b$'s.
The matching $q$-cycle then must also start and end with such sequences.
Then, even if both $a$-sequences at the start of the cycles are completely matched, then the next sequence of $b$'s (of size $p_{4i-2}$ in our case) in the $p$-cycle
can be matched with a corresponding sequence of $b$'s in the $q$-cycle only when this $b$-sequence is part of an $A_i$-segment (the same $i$ as in the $p$-cycle). This is because the size of the sequence of $b$'s that follows a sequence of $a$'s is unique to $A_i$. But then the previous sequence of $a$'s in the $q$-cycle should also be the first part of some $A_i$, which contradicts our assumption that $q$ is a secondary vertex. Similarly, the last part of the $p$-cycle, consisting of a sequence of $b$'s that is preceded by an $a$-sequence, cannot be completely matched with a similar segment at the end of the $q$-cycle. It follows that the middle part of the $p$-cycle stays unfolded, and so the number of cycles cannot be reduced, and the identification of the vertices $p$ and $q$ does not produce a dependent element.

Let us now examine the case of verifying two prime vertices. Here it is clear that we can only match two paths in which both start with an $A_i$-segment (the same $i$) and end with an $A_j$-segment (the same $j$); otherwise, we cannot even match the first sequence of $a$'s or the last sequence of $b$'s.

It remains to look at an identification of two prime vertices $p$ and $q$ where both are incident to an $A_i$-segment or an $A_j^{-1}$-segment (the same $i$ or $j$).
By the preceding discussion, we can look at the $A_i$-segments as single edges, representing free generators for the subgroups they generate (as in Figure~\ref{fig:depseq}(e)). So, let us examine such identifications in $\Gamma(H_4)$ (Figure~\ref{fig:identify}(a)).
Clearly, the identification of vertex 0 (the basepoint $\bp$) with vertex 9 leads to the folding of the $A_4^5$-cycle into the $A_4$-segment and the result, after some more foldings, is seen in Figure~\ref{fig:depseq}(b). That is, $A_4$ is obtained as a dependent element of $H_4$. We claim that all other identification of vertices (other than the ones that result with the same graph as in Figure~\ref{fig:depseq}(b)) increase the number of cycles.
These are shown in Figure~\ref{fig:identify}(b)-(f), where vertices of $\Gamma_A(H_4)$ that are the initial or terminal vertices of an $A_i$-segment (the same $i$ in both vertices) are identified. Not all possible identifications are shown, but the other identifications either lead to graphs that are of the same form as the graphs in Figure~\ref{fig:identify} (possibly with different vertex and edge labels) or are those in which it is clearly seen that they increase the number of cycles from 5 to 6.

The process of constructing the dependence sequence of $H_m$ ends after $m$ steps in the subgroup freely generated by the elements $A_1,\ldots, A_{m+1}$ and then no more dependent elements outside the subgroup exist. It follows that the dependence length of $H_m$ is $m$.
\end{proof}
\end{example}

Despite the existence of these arbitrarily long dependence sequences, the dependence length, $\depl(H)$, of a finitely generated subgroup $H\leqfg F$ is finite, and its dependence closure, $\DepCl(H)$, is finitely generated. In addition, given a set of generators for $H$, one can algorithmically compute a set of generators for all the subgroups in the dependence sequence of $H$.

\begin{proposition}\label{prop:depl}
Let $H\leqfg F(A)$ and let $n$ be the total length of the given generators of $H$. Then
\begin{enumerate}
\item $\depl(H) \leqslant n$ and is computable;
\item $\rk (\DepCl(H))\leqslant \rk(H)$ and one can effectively compute a basis for $\DepCl(H)$ in time $\OO(n^4\log^*(n))$.
\end{enumerate}
\end{proposition}

\begin{proof}
(i) Consider the ascending sequence $H\leqslant \Dep^1(H)\leqslant \Dep^2(H)\leqslant \cdots$ of subgroups of $F(A)$. At each step, when passing from $\Dep^i(H)$ to $\Dep^{i+1}(H)$, we identify at least one pair of vertices of the automaton, which does not result in an increase of rank. When no such pair exits anymore, the resulting dependence-closed subgroup is $\DepCl(H)$. Since the sum of the lengths of the generators is $n$, the process consists of at most $n$ steps, i.e., $\depl(H) \leqslant n$,
	
(ii) The fact that $\rk (\DepCl(H))\leqslant \rk(H)$ follows from Proposition~\ref{prop:rankDep} and from Lemma~\ref{lem:sequence}. By Corollary~\ref{cor:basis}, given a set of generators for $\Dep^i(H)$, we can compute a basis for $\Dep^{i+1}(H)$ in time $\OO(n^3\log^*(n))$. Since the dependence length is at most $n$, a basis for $\DepCl(H)$ can be computed in time $\OO(n^4\log^*(n))$.

In fact, we can compute a set of generators for $\DepCl(H)$ directly from $\Gamma_A(H)$ instead of computing the generators for all the intermediate subgroups $\Dep^i(H)$, and so, accelerate the process. Simply, whenever two vertices $p$ and $q$ are identified in an automaton $\Gamma$ and, after performing all the needed foldings, the new automaton $\Gamma'$ is of degree at most that of $\Gamma$, then we continue the identification of next pairs of elements in $\Gamma'$ instead of $\Gamma$.
\end{proof}

When we are just interested in the question of deciding whether the subgroup $H \leqfg F(A)$ is dependence-closed then we only need to check if $\Dep(H)=H$.

\begin{proposition}
Let $H\leqfg F(A)$ and let $n$ be the total length of the given generators of $H$. Then, it is algorithmically decidable in time $\OO(n^3\log^*(n))$ whether $H$ is dependence-closed or not.
\end{proposition}
\begin{proof}

In order to decide whether $\Dep(H) \neq H$, we only need to check (in the worst case) all pairs of distinct vertices $(p,q)$ of the automaton $\Gamma_A(H)$ and see whether, after identifying $p$ with $q$, the following elementary foldings are all open. By Corollary~\ref{cor:basis}, this can be done in time $\OO(n^3\log^*(n))$.
\end{proof}

\section{Open problems}\label{sec:open}

We gather here some questions that we think are interesting for further research.

\begin{que}[answered by Dario Ascari~\cite{As22}]
Given $H\leqfg F$ and an element $g\in F$ dependent on $H$, is it possible to find a non-trivial $H$-equation of minimal degree that is satisfied by $g$? Is the ideal $I_H(g)$ (of all polynomials $w(x)$, such that $w(g)=1$) generated (as normal subgroup of $H*\langle x\rangle$) by (some of) its polynomials of minimal degree?
\end{que}

\begin{que}
Is the class of 1-auto-fixed subgroups closed under the dependence operator? And the class of 1-gen-endo subgroups? and that of inert subgroups?
\end{que}

\begin{que}
Can one extend the definitions and results in this paper to other classes of groups beside free groups? What are the ``right" definitions of dependence in these cases (see Remark~\ref{rem:other_gps})?
\end{que}

\begin{que}
Is there a finitely presented group $G$ where (some of) the above questions
are algorithmic unsolvable? In the affirmative case, construct such a group.
\end{que}

\begin{que}
Can one develop a similar theory for multi-variate equations over a subgroup $H\leqfg G$? Here, also, one has to find the ``right" definition of dependence.
\end{que}

\section*{Acknowledgments}

The first named author acknowledges the support of the Austrian Science Fund (FWF) project P29355-N35. The second named author acknowledges partial support from the Spanish Agencia Estatal de Investigaci\'{o}n, through grant MTM2017-82740-P (AEI/ FEDER, UE).


\begin{thebibliography}{99}

\bibitem{AJ21} Y. Antol\'{\i}n, A. Jaikin-Zaipirain, ``The Hanna Neumann conjecture for surface groups", preprint available at https://matematicas.uam.es/~andrei.jaikin/preprints/HNsurface.pdf

\bibitem{As22} D. Ascari, ``Ideals of equations for elements in a free group and Stallings folding", arXiv:2207.04759v1, July 2022.

\bibitem{BMV99} G. Baumslag, A. Myasnikov, V. Remeslennikov, ``Algebraic geometry over groups. I. Algebraic sets and ideal theory", \emph{J. Algebra} \textbf{219}(1) (1999), 16--79.

\bibitem{BDM09} D. Bormotov, R. Gilman, A. Myasnikov, ``Solving one-variable equations in free groups Group", \emph{J. Group Theory} \textbf{12} (2009), 317–330.

\bibitem{DV21} J. Delgado, E. Ventura, ``A list of applications of Stallings automata", Transactions on Combinatorics \textbf{11}(3) (2022), 181--235.

\bibitem{DV96} W. Dicks, E. Ventura, ``The group fixed by a family of injective endomorphisms of a free group", \emph{Contemporary Mathematics} \textbf{195}, American Mathematical Society, Providence, RI, 1996.


\bibitem{KM02} I. Kapovich, A. Myasnikov, ``Stallings foldings and subgroups of free groups", \emph{Journal of Algebra} \textbf{248}(2) (2002), 608--668.

\bibitem{Lyn60} R. Lyndon, ``Equations in free groups", \emph{Trans. Amer. Math. Soc.} \textbf{96} (1960), 445--457.

\bibitem{LS01} R. Lyndon, P. Schupp, ``Combinatorial group theory", \emph{Classics in Mathematics}, Springer-Verlag, Berlin, Reprint of the 1977 edition.

\bibitem{Mak82} G. Makanin, ``Equations in a free group", \emph{Izv. Akad. Nauk SSSR Ser. Mat.} \textbf{46}(6) (1982), 1199--1273, 1344.

\bibitem{MKS66} W. Magnus, A. Karrass, D. Solitar, ``Combinatorial group theory: Presentations of groups in terms of generators and relations", Interscience Publishers [John Wiley \& Sons, Inc.], New	York-London-Sydney, 1966.

\bibitem{MSW01} S. Margolis, M. Sapir, P. Weil, ``Closed subgroups in pro-$V$ topologies and the extension problems for inverse automata", \emph{Internat. J. Algebra Comput.} \textbf{11}(4) (2001), 405--445.

\bibitem{MV03} A. Martino, E. Ventura, ``Examples of retracts in free groups that are not the fixed subgroup of any automorphism", \emph{J. Algebra} \textbf{269} (2003), 735--747.

\bibitem{MV04} A. Martino, E. Ventura, ``A description of auto-fixed subgroups of a free group", \emph{Topology} \textbf{43} (2004) 1133--1164.

\bibitem{MVW07} A. Miasnikov, E. Ventura, P. Weil, ``Algebraic extensions in free groups", \emph{Trends Math.}, Birkh\"auser, Basel (2007), 225--253.
	



\bibitem{P14} D. Puder, ``Primitive words, free factors and measure preservation", \emph{Israel J. Math.} \textbf{201}(1) (2014), 25--73.

\bibitem{Raz87} A.A. Razborov, ``On systems of equations in a free group", \emph{Izv. Akad. Nauk SSSR Ser. Mat.} \textbf{48}(4) (1984), 779--832; \emph{Math. USSR-Izv.} \textbf{25}(1) (1985), 115--162.

\bibitem{Raz94} A. Razborov, ``On systems of equations in free groups", \emph{London Math. Soc. Lecture Note Ser.} (Cambridge Univ. Press, Cambridge) \textbf{204} (1995), 269--283.

\bibitem{RVW07} A. Roig, E. Ventura, P. Weil, ``On the complexity of the Whitehead minimization problem", \emph{Internat. J. Algebra Comput.} \textbf{17}(8) (2007), 1611--1634.

\bibitem{Ros93} A. Rosenmann, ``An algorithm for constructing Gr\"{o}bner and free Schreier bases in free group algebras", \emph{J. Symb. Comput.} \textbf{16}(6) (1993), 523--549.

\bibitem{Ros01} A. Rosenmann, ``On rank, root and equations in free groups", \emph{Internat. J. Algebra Comput.} \textbf{11}(3) (2001), 375--390.

\bibitem{Ros13} A. Rosenmann, ``On the intersection of subgroups in free groups: echelon
subgroups are inert", \emph{Groups Complex. Cryptol.} \textbf{5}(2) (2013), 211--221.

\bibitem{RV21} A. Rosenmann, E. Ventura, ``Dependence and algebraicity over subgroups of free groups", arXiv:2107.03154v1, July 2021.

\bibitem{SW08} P.V. Silva, P. Weil, ``On an algorithm to decide whether a free group is a free factor of another", \emph{RAIRO - Theoretical informatics and applications - Informatique th\'eorique et applications} \textbf{42}(2) (2008), 395--414.

\bibitem{Sta83} J. Stallings, ``Topology of finite graphs", \emph{Invent. Math.} \textbf{71}(3) (1983), 551--565.

\bibitem{Tak51} M. Takahasi, ``Note on chain conditions in free groups", \emph{Osaka Math. Journal} \textbf{3}(2) (1951), 221--225.

\bibitem{Tou06} N. Touikan, ``A fast algorithm for Stallings' folding process", International Journal of Algebra and Computation \textbf{16}(6) (2006), 1031--1045.

\bibitem{V97} E. Ventura, ``On fixed subgroups of maximal rank", \emph{Comm. Algebra} \textbf{25} (1997), 3361--3375.
\end{thebibliography}
\end{document}